\documentclass[letterpaper, 11pt]{amsart}
\usepackage{amstext,amsmath,amssymb,amsfonts,orcidlink}%
\usepackage{amsthm}%
\usepackage{xcolor}%
\usepackage{manyfoot}%
\usepackage[notextcomp]{stix}
\usepackage[utf8]{inputenc}
\usepackage[italian, english]{babel}
\usepackage[T1]{fontenc} 
\usepackage{braket,mathtools}
\usepackage{hyperref,csquotes}
\usepackage[english,capitalize]{cleveref}
\usepackage[a4paper,twoside,top=0.8in, bottom=0.7in, left=0.7in, right=0.7in]{geometry}
\usepackage{tikz}
\usetikzlibrary{patterns}

\definecolor{newblue}{rgb}{0.2, 0.3, 0.85}
\hypersetup{colorlinks=true, linkcolor=newblue, citecolor=newblue, urlcolor  = newblue}

\usepackage[maxbibnames=99,backend=biber, sorting=nyt, doi=false, url=false, isbn=false, style=alphabetic]{biblatex}
\numberwithin{equation}{section}
\newcommand{\N}{\mathbb{N}}

\newcommand{\R}{\mathbb{R}}

\newcommand{\ff}{\mathscr{F}}
\newcommand{\g}{\mathscr{G}}
\newcommand{\hh}{\mathcal{H}}
\newcommand{\rr}{\mathscr{R}}
\renewcommand{\epsilon}{\varepsilon}
\renewcommand{\theta}{\vartheta}
\renewcommand{\rho}{\varrho}
\renewcommand{\phi}{\varphi}
\newcommand{\de}{\,{\rm d}}
\renewcommand{\d}{{\rm d}}

\newcommand{\st}{\ensuremath{\ :\ }} 
\newcommand{\mres}{\mathbin{\vrule height 1.6ex depth 0pt width
0.13ex\vrule height 0.13ex depth 0pt width 1.3ex}}
\makeatletter
\def\paragraph{\@startsection{paragraph}{4}%
  \z@\z@{-\fontdimen2\font}%
  {\normalfont\bfseries}}
\makeatother

\theoremstyle{plain}
\newtheorem{theorem}{Theorem}[section]
\newtheorem{proposition}[theorem]{Proposition}
\newtheorem{lemma}[theorem]{Lemma}
\newtheorem{corollary}[theorem]{Corollary}

\theoremstyle{remark}

\newtheorem{remark}[theorem]{Remark}

\theoremstyle{definition}
\newtheorem{definition}[theorem]{Definition}

\begin{document}

\title[Existence and nonexistence of minimizers for classical capillarity problems in presence of ...]{Existence and nonexistence of minimizers for classical capillarity problems in presence of nonlocal repulsion and gravity}

\author{Giulio Pascale}
\address{Dipartimento di Matematica e Applicazioni "Renato Caccioppoli", Universit\'a degli Studi di Napoli "Federico II", via Cintia - Monte Sant'Angelo, 80126 Napoli, Italy} \email{giulio.pascale@unina.it \orcidlink{0000-0003-1680-3425}}

\date{\today}
\subjclass{Primary: 49J40, 49Q10. Secondary: 49Q20, 28A75, 26B30.}
\keywords{Isoperimetric problem, capillarity, Riesz energy, gravitational potential, existence, nonexistence, generalized existence}

\begin{abstract}
We investigate, under a volume constraint and among sets contained in a Euclidean half-space, the minimization problem of an energy functional given by the sum of a capillarity perimeter, a nonlocal interaction term and a gravitational potential energy. 
The capillarity perimeter assigns a constant weight to the portion of the boundary touching the boundary of the half-space.
The nonlocal term is represented by a double integral of a positive kernel $g$, while the gravitational term is represented by the integral of a positive potential $G$. 

We first establish existence of volume-constrained minimizers in the small mass regime, together with several qualitative properties of minimizers. 
The existence result holds for rather general choices of kernels in the nonlocal interaction term, including attractive-repulsive ones.\\
When the nonlocal kernel $g(x)=1/|x|^\beta$ with $\beta \in (0,2]$, we also obtain nonexistence of volume constrained minimizers in the large mass regime.\\
Finally, we prove a generalized existence result of minimizers holding for all masses and general nonlocal interaction terms, meaning that the infimum of the problem is realized by a finite disjoint union of sets thought located at ``infinite distance'' one from the other.

These results stem from an application of quantitative isoperimetric inequalities for the capillarity problem in a half-space.
\end{abstract}

\maketitle

\setcounter{tocdepth}{1} 
\tableofcontents

\section{Introduction}\label{sec:1}
The classical liquid drop model for the atomic nucleus in the Euclidean space $\R^n$, for $n \ge 2$, aims to characterize minimizers of the functional \begin{equation*} P(E) + \int_E\int_E \frac{1}{|y - x|^\beta}\de y \de x \end{equation*} among sets with a given volume, where $0 < \beta < n$ is a given parameter and $P(E)$ denotes the perimeter of $E \subset \R^n$.
There is a clear competition between the two terms in the energy, since the ball at the same time minimizes the perimeter, by the isoperimetric inequality \cite{DeGiorgiIsoperimetrico}, \cite[Theorem 14.1]{MaggiBook}, and maximizes the second term, by the Riesz rearrangement inequality \cite{RieszInequality}, \cite[Theorem 3.7]{LiebLoss}.
The physically relevant case is when $\beta = 1$ and $n = 3$, that is when the second term is the Coulombic energy. 
This case goes back to Gamow's liquid drop model for atomic nuclei \cite{Gamow1930}, subsequently developed by von Weizs\"acker \cite{Weizsacker}, Bohr \cite{Bohr, BohrWheeler}, and many other researchers.
This model is used to explain various properties of nuclear matter \cite{CohenPlasilSwiatecki, CohenSwiatecki, MyersSwiatecki, PelakasisTsamopoulosManolis}, but it also arises in the Ohta-Kawasaki model for diblock copolymers \cite{OhtaKawasaki} and in many other physical situations, see \cite{CareMarch, ChenKhachaturyan, deGennes, EmeryKivelson, GlotzerDiMarzioMuthukumar, KovalenkoNagaev, Mamin, Nagaev, NyrkovaKhokhlovDoi}. 
For a more specific account on the physical background of this kind of problems, we refer to \cite{MuratovTheory}. \par
\vspace{2mm}
In the last decades, the model for general $\beta$ and $n$ has gained renewed interest in mathematics literature, in order to investigate existence and non-existence of minimizers and the minimality of the ball.
In \cite{KnupferMuratovI, KnupferMuratovII}, Kn\"upfer and Muratov proved that balls are the only minimizers in the small mass regime when $n = 2$ and when $3 \le n \le 7$ with $0 < \beta < n - 1$.
At the same time they obtained nonexistence results when $n \ge 2$ and $\beta \in (0, 2)$.
See also the alternative proofs \cite{FrankKillipNam, LuOtto, JulinIsoperimetric} in the case $\beta = 1$ and $n = 3$ and \cite{MuratovZaleski} in the case $n = 2$ with $\beta$ sufficiently small.
Later on, Bonacini and Cristoferi \cite{BonaciniCristoferi} proved existence and uniqueness results for every $n$ and $0 < \beta < n - 1$.
Finally, Figalli, Fusco, Maggi, Millot and Morini \cite{FigalliFuscoMaggiMillotMorini} studied the case $0 < \beta < n$ for every $n$, even replacing the perimeter $P(E)$ by the fractional perimeter $P_s(E)$, $0 < s \le 1$.
We refer to \cite{ChoksiMuratovTopaloglu, NovagaOnoueReview} for a review on the topic and to \cite{ChoksiNeumayerTopaloglu, FrankNonspherical, FrankLiebCompactness, FrankNam, JulinRemark, NovagaOnoueHolder} and references therein for further results on the nonlocal liquid drop model.
A variant of the problem with a constant background has been studied by \cite{AlbertiChoksiOtto, CicaleseSpadaro, ChoksiPeletierI, ChoksiPeletierII, EmmertFrankKonig, FrankLiebPeriodic, KnupferMuratovNovaga}; see also \cite{AcerbiFuscoMorini, AlamaBronsardChoksiTopaloglu, CesaroniNovagaMinkowski, FrankNamVandenbosch, GoldmanMuratovSerfatyI, GoldmanMuratovSerfatyII, MuratovDroplet, NamIonization, OnoueNonexistence, SternbergTopaloglu} for further results on related problems. \par
\vspace{2mm}
In this paper we prove existence and nonexistence results of minimizers in a capillarity context with nonlocal and gravitational terms.
If $E$ is a measurable set in the half-space $\{x_n > 0\} \subset \R^n$, $n \ge 2$, and $\lambda \in (- 1, 1)$, we define the weighted perimeter functional \begin{equation*} P_\lambda(E) := P(E, \R^n \setminus H) - \lambda \hh^{n - 1}(\partial^*E \cap \partial H), \end{equation*} where $\hh^k$, with $k \ge 0$, denotes the $k$-dimensional Hausdorff measure in $\R^n$, $H := \{x_n \le 0\}$ and $\partial^*E$ denotes the reduced boundary of $E$, see \cref{sec:2}.
Interpreting the perimeter as a measure of the surface tension of a liquid drop, the constant $\lambda$ basically represent the relative adhesion coefficient between a liquid drop and the solid walls of the container given by $\{x_n > 0\}$.
If $m > 0$, minimizers for the isoperimetric capillarity problem \begin{equation}\label{minimi:plambda} \inf\{P_\lambda(E) \st E \subset \{x_n > 0\}, \; |E| = m\} \end{equation} are given by suitably truncated balls lying on the boundary of the half-space, see \cite[Theorem 19.21]{MaggiBook}.
More precisely, if $B^\lambda := \{x \in B_1(0) \subset \R^n \st \Braket{x, e_n} > \lambda\}$, $m > 0$ and \begin{equation*} B^\lambda(m) := \frac{m^{\frac{1}{n}}}{|B^\lambda|^{\frac{1}{n}}} (B^\lambda - \lambda \, e_n), \end{equation*} minimizers for~\eqref{minimi:plambda} are sets of the form \begin{equation*} B^\lambda(m, x) := B^\lambda(m) + x, \end{equation*} with $x \in \{x_n = 0\}$, see also \cite[Fig. 1]{PascalePozzettaQuantitative} \\
If $g : \R^n \setminus \{0\} \to (0, \infty)$, we define the Riesz-type potential energy \begin{equation*}\label{riesz:definition} \rr(E) := \int_E\int_E g(y - x) \de y \de x. \end{equation*} 
Finally, given a function $G : (0, \infty) \to (0, \infty)$, we define the gravity-type potential energy \begin{equation*} \g(E) := \int_EG(x_n) \de x. \end{equation*}
If $m > 0$ and we denote \begin{equation*} \ff^\lambda(E) := P_\lambda(E) + \rr(E) + \g(E), \end{equation*} we consider the minimization problem \begin{equation}\label{minimum:problem} \inf \{\ff^\lambda(E) \st E \subset \{x_n > 0\}, \, |E| = m\}. \end{equation} \par

\vspace{2mm}

In the context of minimization of energies with general Riesz-type potential in the Euclidean space $\R^n$, Novaga and Pratelli in \cite{NovagaPratelli} showed the existence of minimizers in a generalized sense.
Later on, Carazzato, Fusco and Pratelli in \cite{CarazzatoFuscoPratelli} showed that the ball is the unique minimizer in the small mass regime when the nonlocal kernel $g$ is radial and decreasing.
Pegon in \cite{PegonLarge} proved that, if the kernel $g$ decays sufficiently fast at infinity and if the volume is sufficiently large, then minimizers exist and converge to a ball as the volume goes to infinity.
Then, Merlet and Pegon \cite{MerletPegon} proved that in the planar case minimizers are actually balls in the large mass regime.
In \cite{NovagaOnoueFractional}, Novaga and Onoue obtained existence of minimizers for any volume and convergence to a ball as volume goes to infinity,  if the Riesz potential decays sufficiently fast and even if the perimeter $P(E)$ is replaced by the fractional perimeter $P_s(E)$, $0 < s < 1$. 
We refer to \cite{BessasNovagaOnoue, CesaroniNovagaNonlocal, MelletWu, MuratovSimon, RigotEnsemble} and references therein for further results on variational problems involving nonlocal energies. \par
\vspace{2mm}
The first result in this paper is an existence result in the capillarity context and in the small mass regime, together with a bound on the Fraenkel asymmetry \begin{equation*} \alpha_\lambda(E) := \inf \left\{\frac{|E \Delta B^\lambda(m, x)|}{m} \st x \in \{x_n = 0\}\right\}, \qquad{\rm for}\, {\rm any} \; E \subset \{x_n > 0\} \; {\rm with} \; |E| = m, \end{equation*} and some qualitative properties of volume constrained minimizers. 
At the same time, under suitable conditions on the potential energies, existence result extends to all masses.

\begin{theorem}\label{theorem:1}
Let $g$ be a $\rr$-admissible $q$-growing function, $q \ge 0$, and let $G$ be a $\g$-admissible function.
There exists a mass $\bar m = \bar m(n, \lambda, g, G, q) > 0$ such that, for every $m \in (0, \bar m)$, there exists a minimizer $E$ of $\ff^\lambda$ in the class \begin{equation*} \mathcal{A}_m := \{\Omega \subset \R^n \setminus H \,{\rm measurable} \st |\Omega| = m\} \end{equation*} and it satisfies \begin{equation*} \alpha_\lambda(E) \le c(n, \lambda, g, G) \, m^{\frac{1}{2n}}. \end{equation*}
Moreover, if $g$ is also infinitesimal, minimizers are indecomposable and, if in addition $g$ is symmetric, minimizers are essentially bounded.

Furthermore, if $g$ is also $0$-growing, infinitesimal and symmetric and $G$ is coercive, minimizers have no holes, i.e., if $E$ is a minimizer of $\ff^\lambda$ in $\mathcal{A}_m$, there is no set $F \subset \R^n \setminus (H \cup E)$ with $|F| > 0$ such that \begin{equation*} P_\lambda(E) = P_\lambda (E \cup F) + P(F, \R^n \setminus H) + \lambda \hh^{n - 1}(\partial^*F \cap \partial H). \end{equation*}

Finally, if $g$ is $\rr$-admissible and coercive and $G$ is $\g$-admissible and coercive, there exists a minimizer of $\ff^\lambda$ in $\mathcal{A}_m$ for any $m > 0$.
\end{theorem}

Let us make some comments on the definitions present in \cref{theorem:1}, while referring to \cref{sec:2} for their precise enunciation.
The ``admissibility'' requirements on the kernels just refer to some necessary integrability conditions.
The infinitesimality of $g$ and the coercivity of $g$ and $G$ concern the behavior of these functions as the variable diverge, while the symmetry of $g$ is referred to the symmetry with respect to the origin.
The $q$-growing property is satisfied by rather general nonlocal interaction terms, not only by repulsive ones. 
Indeed, we point out that classical radial decreasing kernels are $0$-growing, but at the same time attractive-repulsive kernels of the type \begin{equation}\label{eq:AttractiveRepulsive} g(x) = |x|^{\beta_1} + \frac{1}{|x|^{\beta_2}},\qquad \beta_1 > 0, \quad \beta_2 \in (0, n), \end{equation} are $q$-growing for any $q\ge \beta_1$, even if they diverge positively as $|x| \to + \infty$, see \cref{def:KernelgPiccolo}  and \cref{attractive:repulsive}. In particular attractive-repulsive kernels as in \eqref{eq:AttractiveRepulsive} represent a possible choice in the definition of $\ff^\lambda$ in \cref{theorem:1}.
Minimization problems for attractive-repulsive functionals have been widely studied in the last years.
Existence and nonexistence results are addressed in \cite{BurchardChoksiTopaloglu, FrankLiebLiquidSolid, FrankLiebProof}, while stability and uniqueness of minimizers have been respectively studied in \cite{BonaciniCristoferiTopaloglu, LopesUniqueness}.
We refer to \cite{CarazzatoNote, CarazzatoPratelli, CarazzatoPratelliTopaloglu, CarrilloDelgadinoMellet} for further results about analogous problems. 

We remark that, by a symmetry argument, analyzing the Euler-Lagrange equation of problem~\eqref{minimum:problem}, it is possible to verify that the sets $B^\lambda(m, x)$ are not volume constrained minimizers of $\ff^\lambda$; actually, the isoperimetric bubbles $B^\lambda(m, x)$ are not even volume constrained critical points of $\ff^\lambda$.
It is left as a future project to study quantitative properties of minimizers to ~\eqref{minimum:problem}, such as the proximity of minimizers from bubbles $B^\lambda(m, x)$ in terms of the smallness of the mass.
Related results, concerning liquid drops or crystals lying in equilibrium under the action of a potential energy, are addressed in \cite{FigalliMaggiShape, MaggiMihaila}.

\vspace{2mm}
For large masses and for suitable choices of repulsive kernels $g$, the repulsive interaction dominates and the variational problem in \cref{theorem:1} does not admit a minimizer.

\begin{theorem}\label{theorem:2}
Let \begin{equation*} g(x) = \frac{1}{|x|^\beta}, \qquad 0 < \beta < n, \, x \in \R^n \setminus \{0\} \end{equation*} and let $G$ be $\g$-admissible.
For every $\beta \in (0, 2]$, there exists $\tilde m > 0$, depending on $n$, $\lambda$, $\beta$, $G$, such that for all $m \ge \tilde m$ the minimization problem \begin{equation*} \inf\{ \ff^\lambda(E) \st E \subset \R^n \setminus H, |E| = m \} \end{equation*} has no minimizers. \end{theorem}
Therefore, for a general repulsive kernel $g$, existence may fail for masses large enough, since minimizers tend to split in two or more components which then move apart one from the other in order to decrease the nonlocal energy.
To capture this phenomenon, it is convenient to introduce a generalized energy defined as \begin{equation*}\tilde\ff^\lambda(E) := \inf_{h \in \N}\tilde\ff^\lambda_{h}(E), \end{equation*} where \begin{equation*} \tilde\ff^\lambda_{h}(E) := \inf \left\{\sum_{i = 1}^h \ff^\lambda(E^i) \st E = \bigcup_{i = 1}^h E^i, E^i \cap E^j = \emptyset \quad {\rm for} \; 1 \le i \neq j \le h\right\}. \end{equation*}
Note that in this functional the interaction between different components is not evaluated, which corresponds to consider them ``at infinite distance'' one from the other.
By considering $\tilde \ff^\lambda$ instead of $\ff^\lambda$, we can prove the following generalized existence result.
\begin{theorem}\label{theorem:3} Let $g$ be a $\rr$-admissible $q$-growing function, $q \ge 0$, and let $G$ be a $\g$-admissible function.
For every $m > 0$ there exists a minimizer of $\tilde\ff^\lambda$ in the class \begin{equation*} \mathcal{A}_m = \left\{\Omega \subset \R^n\setminus H \; {\rm measurable} \st |\Omega| = m\right\}. \end{equation*} 
More precisely, there exist a set $E \in \mathcal{A}$ and a subdivision $E = \cup_{j = 1}^h E^j$, with pairwise disjoint sets $E^j$, such that \begin{equation*} \tilde\ff^\lambda(E) = \sum _{j = 1}^h \ff^\lambda(E^j) = \inf \left\{\tilde\ff^\lambda(\Omega) \st \Omega \in \mathcal{A}_m\right\}. \end{equation*} 
Moreover, for every $1 \le j \le h$, the set $E^{j}$ is a minimizer of both the standard and the generalized energy for its volume, i.e. \begin{equation*} \tilde\ff^\lambda(E^{j}) = \ff^\lambda(E^{j}) = \min \left\{\tilde\ff^\lambda(\Omega) \st \Omega \subset \R^n\setminus H, \, |\Omega| = |E^{j}|\right\}. \end{equation*} \end{theorem}
\begin{remark} We note that, if $g$ is infinitesimal, then \begin{equation*} \inf\left\{\tilde\ff^\lambda(\Omega) \st \Omega \in \mathcal{A}_m\right\} = \inf\left\{\ff^\lambda(\Omega) \st \Omega \in \mathcal{A}_m\right\}. \end{equation*}
The proof can be easily adapted by \cite[Lemma 3.4]{NovagaPratelli}, with attention given to translating the components of a partition without changing the $n$-th component.
In this case, every minimizer of $\ff^\lambda$ is also a minimizer of $\tilde\ff^\lambda$.
It is left as a future project to study connections between minimizers and generalized minimizers when $g$ is not infinitesimal. \end{remark}
\vspace{2mm}

\paragraph{Strategy of the proof and comments}
The proof of \cref{theorem:1} is divided into several steps.
In the spirit of \cite{KnupferMuratovII}, the existence of minimizers in the small mass regime follows by the direct method of the calculus of variations, see \cref{simple:existence}, once we show that for sufficiently small mass every minimizing sequence of the energy may be replaced by another minimizing sequence where all sets have uniformly bounded diameter, see \cref{lemma5.1:km}.
We remark that we heavily use the quantitative isoperimetric inequality for the capillarity problem proved in \cite{PascalePozzettaQuantitative}, which estimates the Fraenkel asymmetry of a competitor with respect to the optimal sets in terms of the energy deficit.
Note that it is unclear at the moment how to apply stronger isoperimetric inequalities of Fuglede-type \cite{CicaleseLeonardi, Fuglede} for nearly spherical sets in the present capillarity framework; instead, stronger isoperimetric inequalities have been used as fundamental tools, for example, in \cite{AcerbiFuscoMorini, BonaciniCristoferi, CarazzatoFuscoPratelli}.
In fact, the classical Fuglede's method relies on the precise knowledge of the eigenvalues of the Laplace-Beltrami operator, which is not available for $P_\lambda$ on optimal sets for generic $\lambda \in (-1, 1)$.
Moreover, in our case it is in general not possible to globally parametrize $C^1$-close boundaries one on the other as normal graphs.\par
\vspace{2mm}
The boundedness result in \cref{theorem:1} follows once we show that minimizers enjoy uniform density estimates at boundary points.
In order to do so, we prove that, under suitable conditions on the Riesz potential, minimizers are $(K, r_0)$-quasiminimal sets for all masses, see \cref{quasiminimal:set} and \cref{lemma:quasiminimal}.
Indeed quasiminimal sets have well-known topological regularity properties (\cref{density:quasiminimal}), which easily guarantee boundedness, see \cref{teo:limitato}.
Note that the lack of symmetry of the problem, due to the presence of gravitational potential and the fact that ambient space is a half-space, forces us to deal with the vertical direction in a separate way, see \cref{vertical:boundedness}. \par
\vspace{2mm}
The absence of holes is based on the combination of some techniques from \cite{KnupferMuratovII} and \cite{NovagaPratelli}. 
We firstly prove some density estimates which improve, under suitable hypotheses on $g$, the analogous estimates for quasiminimal sets, by providing bounds independent of the minimizer, see \cref{lemma4.3:km}. 
In fact, this allows to prove the boundedness in the vertical direction with a bound independent of the minimizer, see \cref{vertical:boundedness2}, and to obtain absence of holes arguing by contradiction, see \cref{teo:nobuchi}. \par
\vspace{2mm}

The proof of \cref{theorem:2} is based on the combination of some techniques from \cite{FrankNam} and \cite{KnupferMuratovII} and exploits some estimates on the diameter and the nonlocal potential energy of minimizers.
We remark that the range of the exponent $\beta$ in \cref{theorem:2} is the same as the analogous nonexistence results in the classical setting \cite{ChoksiNeumayerTopaloglu, FrankKillipNam, FrankNam, KnupferMuratovII, LuOtto}. \par
\vspace{2mm}

The proof of \cref{theorem:3} is inspired by \cite{NovagaPratelli} and exploits the isoperimetric inequality for the capillarity functional $P_\lambda$. In our case the argument must be modified to take into account the presence of the gravitational energy and, as before, estimates in the vertical direction must be treated separately. We remark that also the possible choices for the kernels $g$ in our \cref{theorem:3} allow for more freedom than those considered in \cite{NovagaPratelli}.
\vspace{2mm}
\paragraph{Organization} In \cref{sec:2} we collect definitions and facts on sets of finite perimeter and capillarity functional.
In \cref{sec:3} we begin to prove \cref{theorem:1}, in particular we prove existence, boundedness and indecomposability of minimizers.
In \cref{sec:4} we prove \cref{theorem:2}, together with absence of holes in minimizers in the small mass regime, thus completing the proof of \cref{theorem:1}. In \cref{sec:5} we prove \cref{theorem:3}.

\section{Preliminaries}\label{sec:2}

\begin{center}
\emph{From now on and for the rest of the paper we assume that $\lambda \in (-1,1)$ and $n\in \N$ with $n\ge 2$ are fixed.}    
\end{center}

\medskip
\noindent\textbf{List of symbols.}

\begin{itemize}
    \item $|\cdot|$ denotes Lebesgue measure in $\R^n$.
    
    \item $B^\lambda = \{x \in B : \Braket{x, e_n} > \lambda\}$.
    
    \item $B^\lambda(m) := \frac{m^{\frac{1}{n}}}{|B^\lambda|^{\frac1n}}(B^\lambda - \lambda e_n)$, for any $m>0$.
    
    \item $B^\lambda(m,x) := B^\lambda(m)  +x$, for any $x \in \{x_n=0\}$. In particular $B^\lambda(m)=B^\lambda(m,0)$.

    \item $c(\cdot), C(\cdot)$ denote strictly positive constants, that may change from line to line.
    
    \item $H:=\{x_n\le0\}$.

    \item $\hh^d$ denotes $d$-dimensional Hausdorff measure in $\R^n$, for $d\ge0$.
    
    \item $Q_{\frac{r}{2}}$ denotes a generic cube in $\R^n$ of side $r > 0$.
    
    \item $R_\lambda := \max\left\{\sqrt{1 - \lambda^2}, 1 - \lambda\right\}$.
\end{itemize}

\subsection{Sets of finite perimeter}\label{sec:SetsFinitePerimeter}
We recall basic definitions and properties regarding sets of finite perimeter, referring to \cite{AmbrosioFuscoPallara, MaggiBook} for a complete treatment on the subject.
The perimeter of a measurable set $E\subset\R^n$ in an open set $A\subset \R^n$ is defined by
\begin{equation}\label{eq:DefPerimetro}
P(E,A):=\sup\left\{\int_E {\rm div}\, T(x) \de x \st T \in C_c^1(A;\R^n), \,\, \|T\|_\infty \le1\right\}. \end{equation}
Denoting $P(E):=P(E,\R^n)$, we say that $E$ is a set of finite perimeter if $P(E)<+\infty$.
In such a case, the characteristic function $\chi_E$ has a distributional gradient $D\chi_E$ that is a vector-valued Radon measure on $\R^n$ such that \begin{equation*}
\int_E {\rm div}\, T(x) \de x = - \int_{\R^n} T  \de D\chi_E , \quad \forall T \in C_c^1(\R^n; \R^n). \end{equation*} 
It can be proved that the set function $P(E,\cdot)$ defined in \eqref{eq:DefPerimetro} is the restriction of a nonnegative Borel measure to open sets. The measure $P(E,\cdot)$ coincides with the total variation $|D\chi_E|$ of the distributional gradient, and it is concentrated on the reduced boundary
\begin{equation*}
\partial^*E : = \left\{ x \in {\rm spt} |D\chi_E| \st \exists\, \nu^E(x):=-\lim_{r\to0} \frac{D\chi_E(B_r(x))}{|D\chi_E(B_r(x))|} \text{ with } |\nu^E(x)|=1 \right\}.
\end{equation*}
Introducing the sets of density $t\in[0,1]$ points for $E$ defined by
\begin{equation*}
    E^{(t)} := \left\{ x \in \R^n \st \lim_{r\to0} \frac{|E \cap B_r(x)|}{|B_r(x)|}= t \right\},
\end{equation*}
we have that the reduced boundary coincides both with $\R^n\setminus(E^{(1)} \cup E^{(0)})$ and with the set $E^{(1/2)}$ up to $\hh^{n-1}$-negligible sets. The vector $\nu^E$ is called the generalized outer normal of $E$.
Moreover $P(E,\cdot) = \hh^{n-1}\mres \partial^*E$, and the distributional gradient can be written as $D\chi_E = - \nu^E \hh^{n-1}\mres \partial^*E$.

\subsection{Preliminary results on the capillarity functional}
We recall basic properties on the functional $P_\lambda$.
Let $E \subset \R^n \setminus H$ with measure $|E| = m \in (0, + \infty)$. 
Note that \begin{equation*} P_\lambda(E) = \int_{\partial^*E \setminus H} 1 - \lambda \Braket{e_n, \nu^E} \de\hh^{n - 1}; \end{equation*} in particular $P_\lambda(E) > 0$ \cite[Remark 2.2, Corollary 2.5]{PascalePozzettaQuantitative}. \newline
The sets $B^\lambda(m, x)$ uniquely minimize the problem \begin{equation*} \inf \{P_\lambda(E) \st E \subset \R^n, \, |E| = m\}, \end{equation*} and this is encoded in the isoperimetric inequality \cite[Theorem 19.21]{MaggiBook}, \cite[Theorem 3.3]{PascalePozzettaQuantitative} \begin{equation}\label{isoperimetric:inequality} P_\lambda(E) \ge n |B^\lambda|^{\frac{1}{n}} m^{\frac{n - 1}{n}}. \end{equation}
The previous isoperimetric inequality can be strengthened in a quantitative version. 
We define the Fraenkel asymmetry \begin{equation*} \alpha_\lambda(E) := \inf \left\{\frac{|E \Delta B^\lambda(m, x)|}{m} \st x \in \{x_n = 0\}\right\}, \end{equation*} and the isoperimetric deficit \begin{equation*} D_\lambda(E) := \frac{P_\lambda(E) - P_\lambda(B^\lambda(m))}{P_\lambda(B^\lambda(m))}, \end{equation*} for any $E \subset \{x_n > 0\}$ with volume $|E| = m$.
Then the following sharp quantitative isoperimetric inequality holds
\begin{theorem}[\cite{PascalePozzettaQuantitative}] There exists a constant $c(n, \lambda) > 0$ such that for any measurable set $E \subset \R^n \cap \{x_n > 0\}$ with finite measure there holds \begin{equation}\label{quantitativeisoperimetric:inequality} \alpha_\lambda(E)^2 \le c(n, \lambda) D_\lambda(E). \end{equation} \end{theorem}

\subsection{Definitions} 
We provide some definitions for the kernels of Riesz-type potential $\rr$ and gravity-type potential $\g$.
In particular, in the following \cref{def:KernelgPiccolo} and \cref{gadmissible:2} we will impose pointwise requirements on the functions $g$, $G$, i.e., it is understood that we fix pointwise defined representatives for the functions $g$, $G$.
\begin{definition}\label{def:KernelgPiccolo} A function $g : \R^n \setminus \{0\} \to (0, \infty)$ is $\mathscr{R}$-admissible if $g \in L^1_{loc}(\R^n \setminus \{0\})$ and $\mathscr{R}(B_1) < \infty$. \\
A $\rr$-admissible function $g : \R^n \setminus \{0\} \to (0, \infty)$ is $q$-growing, for some $q \in [0, \infty)$, if for every $x \in \R^n \setminus \{0\}$ and every $\alpha > 1$ it holds \begin{equation*} g(\alpha x) \le \alpha^q g(x). \end{equation*}
A $\rr$-admissible function $g : \R^n \setminus \{0\} \to (0, \infty)$ is infinitesimal if \begin{equation*} \lim_{|x| \to + \infty} g(x) = 0. \end{equation*}
A $\rr$-admissible function $g : \R^n \setminus \{0\} \to (0, \infty)$ is symmetric if \begin{equation*} g(- x) = g(x) \qquad \forall x \in \R^n \setminus \{0\}. \end{equation*}
A $\rr$-admissible function $g : \R^n \setminus \{0\} \to (0, \infty)$ is coercive if \begin{equation*} g(x) \to +\infty \qquad {\rm as}\; |x| \to + \infty. \end{equation*}
Given two measurable sets $L$, $M \subset \R^n \setminus H$ we let \begin{equation*} \mathscr{R}(L, M) := \int_L \int_M g(y - x) \de y \de x. \end{equation*}\end{definition}

\begin{remark}
The functions $\frac{1}{|x|^\beta}$, for $\beta \in (0, n)$, are $\rr$-admissible, $0$-growing, infinitesimal and symmetric. 
\end{remark}
\begin{remark}\label{attractive:repulsive} The attractive-repulsive kernels $|x|^{\beta_1} + \frac{1}{|x|^{\beta_2}}$, for $\beta_1 > 0$ and $\beta_2 \in (0, n)$, are $\rr$-admissible $\beta_1$-growing symmetric functions.
At the same time they diverge positively as $|x| \to + \infty$. \end{remark}

\begin{definition}\label{gadmissible:2} A function $G : (0, \infty) \to (0, \infty)$ is $\g$-admissible if $G \in L^1_{loc}(0, \infty)$, \begin{equation}\label{sup:Gbis} \sup_{t \in (0, 2)} G(t) < \infty, \end{equation} and \begin{equation}\label{crescita:Gbis} G(\alpha t) \le \alpha^n G(t), \qquad \forall \alpha >1, \, t > 0. \end{equation} 
A $\g$-admissible function $G : (0, \infty) \to (0, \infty)$ is coercive if \begin{equation*} G(t) \to +\infty \qquad {\rm as}\; t \to + \infty. \end{equation*} \end{definition}

\begin{remark} The identity function $G(t)=t$ on $(0, + \infty)$ is a $\g$-admissible function. \end{remark}

\begin{remark}\label{global:growthbis} Conditions~\eqref{sup:Gbis} and~\eqref{crescita:Gbis} easily imply \begin{equation*} G(t) = G(t \cdot 1) \le t^n G(1) \le c(G) \, t^n, \qquad \forall t > 1. \end{equation*} \end{remark}

\section{Existence of minimizers for small masses}\label{sec:3}
\subsection{Existence}
The goal of this Section is to prove the following 
\begin{theorem}\label{simple:existence} Let $g$ be a $\rr$-admissible $q$-growing function and let $G$ be a $\g$-admissible function.
There exists a mass $\bar m = \bar m(n, \lambda, g, G, q) > 0$ such that, for all $m \in (0, \bar m)$, there exists a minimizer of $\ff^\lambda$ in the class \begin{equation*} \mathcal{A}_m := \{\Omega \subset \R^n \setminus H \,{\rm measurable} \st |\Omega| = m\}. \end{equation*} \end{theorem}

We begin by proving some preparatory lemmas, which estimate the energy of some competitors.

\begin{lemma}\label{lemma2.3:np} Let $g$ be $\rr$-admissible and $G$ be $\g$-admissible. 
There exists a constant $c = c(n, \lambda, g,$ $G)$ such that \begin{equation*} \rr(B^\lambda(m)) \le c\,m, \qquad \g(B^\lambda(m)) \le c \, m \end{equation*} for every $0 < m \le |B^\lambda|$.
\end{lemma}

\begin{proof} Let us denote by $\bar Q_l \subset \R^n \setminus H$, with $l > 0$, the cube $[-l, l] \times \dots \times [-l, l] \times [0, 2l]$.
For any $N \in \N$ the cube $\bar Q_1$ is the essential union of $(2N)^n$ disjoint isometric cubes $Q_{\frac{1}{2N}}^i$ of side $1/N$.
If $\bar Q_{\frac{1}{2N}} \subset \bar Q_1$ is the cube $\left[-\frac{1}{2N}, \frac{1}{2N}\right] \times \dots \times \left[-\frac{1}{2N}, \frac{1}{2N}\right] \times \left[0, \frac{1}{N}\right]$, evidently \begin{equation*} \rr(\bar Q_1) \ge \sum_{i=1}^{(2N)^n}\rr(Q_{\frac{1}{2N}}^i) = (2N)^n \rr(\bar Q_{\frac{1}{2N}}). \end{equation*}
Moreover \begin{equation*} 2^n\sup_{(0, 2)}G = |\bar Q_1| \sup_{(0, 2)}G = (2N)^n \frac{|\bar Q_1|}{(2N)^n}\sup_{(0, 2)}G = (2N)^n \int_{\bar Q_{\frac{1}{2N}}}\sup_{(0, 2)}G\de x \ge (2N)^n \g(\bar Q_{\frac{1}{2N}}). \end{equation*}
For any $0 < r \le 1$ we denote by $N$ the integer part of $\frac{1}{2r}$, so that $(2r)^{-1} \le 2N \le r^{- 1}$.
The above estimates, together with $4rN \ge 1$, imply that \begin{equation*} \rr(\bar Q_r) \le (4r)^n N^n\rr(\bar Q_{\frac{1}{2N}}) \le 2^n \rr(\bar Q_1) r^n \end{equation*} and \begin{equation*} \g(\bar Q_r) \le (4r)^n N^n\g(\bar Q_{\frac{1}{2N}}) \le 4^n \left(\sup_{(0, 2)}G\right)r^n. \end{equation*}
If $r = \frac{m^{\frac{1}{n}}}{|B^\lambda|^{\frac{1}{n}}} \le 1$, since $B^\lambda(|B^\lambda|) \subset \bar Q_1$, we get $B^\lambda(m) \subset \bar Q_r$ and we conclude that \begin{equation*} \rr(B^\lambda(m)) \le \rr(\bar Q_r) \le c r^n \le c m \end{equation*} and \begin{equation*} \g(B^\lambda(m)) \le \g(\bar Q_r) \le c r^n \le c m. \end{equation*} \end{proof}

\begin{corollary}\label{theorem3.4:km} Let $g$ be $\rr$-admissible and infinitesimal and let $G$ be $\g$-admissible.
For every $m \ge 1$ there exists $E \subset \R^n \setminus H$ with $|E| = m$ such that $\ff^\lambda(E) \le c m$ for some $c$ depending on $n$, $\lambda$, $g$ and $G$. \end{corollary}

\begin{proof}
Let us consider the set $E$ given by a collection of $N \ge 1$ spherical caps $\left\{B^\lambda(v, x_i)\right\}_{1 \le i \le N}$ of equal volume $v$ and with centers located at $x_i = i R e_1$, $i = 1, \dots, N$, with $R$ large enough so that $B^\lambda(v, x_i)$ are pairwise disjoint.
We  choose the number $N$ as the smallest integer for which the volume of each spherical cap does not exceed $\min \left\{1, |B^\lambda|\right\}$. 
In particular $Nv = m$ and $N = \left\lceil \frac{m}{\min\{1, |B^\lambda|\}} \right\rceil$.
Note that, by \cite[Lemma 3.1]{PascalePozzettaQuantitative}, since $v \le 1 \le m = |E|$, \begin{equation*} \begin{split} P_\lambda(E) = P_\lambda\left(\cup_{i=1}^N B^\lambda(v, x_i)\right) & = \sum_{i=1}^N P_\lambda(B^\lambda(v, x_i)) = c(n, \lambda) N v^{\frac{n - 1}{n}} \le c(n, \lambda) \left(\frac{m}{\min\{1, |B^\lambda|\}} + 1\right) v^{\frac{n - 1}{n}} \\ & \le c(n, \lambda) \, (m \, v^{\frac{n - 1}{n}} + v^{\frac{n - 1}{n}} )\le c(n, \lambda) \, (m\, 1^{\frac{n - 1}{n}} + m \, 1^{\frac{n - 1}{n}} ) = c(n, \lambda) m. \end{split} \end{equation*}
Moreover, let $R$ be so large that $g(x - y) < \frac{1}{N}$ for every $x \in B^\lambda(v, x_j)$, $y \in B^\lambda(v, x_k)$ with $j \neq k$.
Then, by \cref{lemma2.3:np}, since $v \le 1 \le m = |E|$, \begin{equation*} \begin{split} \rr (E) & = \int_{\bigcup_{i = 1}^N B^\lambda(v, x_i)} \int_{\bigcup_{i = 1}^N B^\lambda(v, x_i)}g(y - x) \de y \de x \\ & = \sum_{i = 1}^N \rr(B^\lambda(v, x_i)) + \sum_{i = 1}^N \int_{B^\lambda(v, x_i)} \int_{B^\lambda(v, x_1) \cup \dots \cup \widehat{B^\lambda(v, x_i)}\cup \dots \cup B^\lambda(v, x_N)} g(y - x) \de y \de x \\ & \le c(n, \lambda, g, G) N v + N \frac{1}{N} v (m - v) \le c(n, \lambda, g, G) m + (m - v) \le c(n, \lambda, g, G) m, \end{split} \end{equation*}
where the $B^\lambda(v, x_1) \cup \dots \cup \widehat{B^\lambda(v, x_i)}\cup \dots \cup B^\lambda(v, x_N)$ denotes union over all the bubbles except for $B^\lambda(v, x_i)$.
Finally, by \cref{lemma2.3:np}
\begin{equation*}
\begin{split} \ff^\lambda(E) & = P_\lambda(E) + \rr(E) + \g (E) \\ & \le c(n, \lambda) |E| + c(n, \lambda, g, G) |E| + \sum_{i = 1}^N \g(B^\lambda(v, x_i)) \\ & \le c(n, \lambda) |E| + c(n, \lambda, g, G) |E| + c(n, \lambda, g, G) N v = c(n, \lambda, g, G) m .\end{split} \end{equation*} \end{proof}

\begin{lemma}\label{lemma3.1:np} Let $E \subset \R^n \setminus H$ be a set of finite perimeter. Let $g$ be $\mathscr{R}$-admissible and $q$-growing and let $G$ be $\g$-admissible. 
If $\alpha > 1$, then \begin{equation*} \ff^\lambda(\alpha E) \le \alpha^{2n + q} \ff^\lambda(E). \end{equation*} \end{lemma}
\begin{proof} Note that, if $E \subset \R^n \setminus H$, then $\alpha E \subset \R^n \setminus H$.
Since $\alpha > 1$, by the positivity of $P_\lambda$ we get \begin{equation*} \begin{split} P_\lambda(\alpha E) & = \alpha^{n - 1} P_\lambda(E) \le \alpha^{2n + q} P_\lambda(E). \end{split} \end{equation*}
Since $g$ is $q$-growing, we have
\begin{equation*} \begin{split} \rr(\alpha E) & = \int\int_{(\alpha E)^2} g(y- x) \de y \de x \\ & = \alpha^{2n} \int\int_{E \times E} g(\alpha(y - x))\de y \de x \\ & \le \alpha^{2n + q}\int\int_{E \times E} g(y - x) \de y \de x = \alpha^{2n + q}\rr(E). \end{split} \end{equation*} 
Finally, by~\eqref{crescita:Gbis} we get
\begin{equation*} \begin{split} \int_{\alpha E} G(x_n) \de x & = \alpha^n \int_E G(\alpha x_n)\de x \le \alpha^{2n + q} \int_E G(x_n) \de x. \end{split} \end{equation*} \end{proof}

The following lemma allows to suitably localize minimizing sequences with sufficiently small volume.

\begin{lemma}\label{lemma5.1:km} Let $g$ be $\rr$-admissible and $q$-growing and let $G$ be $\g$-admissible. 
There exists $\bar m > 0$, depending on $n$, $\lambda$, $g$, $G$ and $q$, such that, for every $m \in (0, \bar m)$ and every set of finite perimeter $F \subset \R^n \setminus H$ with $|F| = m$, there exists a set of finite perimeter $L$ with \begin{equation}\label{5.1:km} \ff^\lambda(L) \le \ff^\lambda(F) \qquad {\rm and} \qquad L \subset \bar Q_1 := [-1, 1] \times\dots\times[-1, 1] \times[0, 2]. \end{equation} \end{lemma}

\begin{proof}
Throughout the proof we will assume that $\bar m < \frac{|B^\lambda|}{4^n}$.
If $\ff^\lambda(B^\lambda(m)) \le \ff^\lambda(F)$, then the assertion of the lemma is proved by choosing $L = B^\lambda(m)$.
Then we can assume, by \cref{lemma2.3:np}, that \begin{equation}\label{4.5:km} \ff^\lambda(F) < \ff^\lambda(B^\lambda(m)) \le c(n, \lambda, g, G) \max\left\{m, m^{\frac{n - 1}{n}}\right\}. \end{equation}
By \cref{lemma2.3:np} \begin{equation}\label{5.2:km} \begin{split} D_\lambda(F) & = \frac{c(n, \lambda)}{m^{\frac{n - 1}{n}}}(P_\lambda(F) - P_\lambda(B^\lambda(m))) \\ & \le \frac{c(n, \lambda)}{m^{\frac{n - 1}{n}}} ([\rr(B^\lambda(m)) - \rr(F)] + [\g(B^\lambda(m)) - \g(F)]) \\ & \le \frac{c(n, \lambda)}{m^{\frac{n - 1}{n}}} (\rr(B^\lambda(m)) + \g(B^\lambda(m))) \\ & \le c(n, \lambda, g, G)\,\left(\frac{m}{|B^\lambda|}\right)^{\frac{1}{n}}. \end{split} \end{equation} 
By the quantitative isoperimetric inequality~\eqref{quantitativeisoperimetric:inequality} \begin{equation}\label{5.2:kmbis} \alpha_\lambda(F) \le c(n, \lambda) \sqrt{D_\lambda(F)} \le c(n, \lambda, g, G) \left(\frac{m}{|B^\lambda|}\right)^{\frac{1}{2n}} \end{equation} and, after a suitable translation, \begin{equation*} |B^\lambda(m) \Delta F| \le c(n, \lambda, g, G) \, \left(\frac{m}{|B^\lambda|}\right)^{1 + \frac{1}{2n}}. \end{equation*}
Since $|F| = m = |B^\lambda(m)|$ we also have \begin{equation*} |B^\lambda(m) \Delta F| = 2 |F\setminus B^\lambda(m)| \end{equation*} and \begin{equation}\label{5.3:km} |F \setminus B^\lambda(m)| \le c(n, \lambda, g, G) \left(\frac{m}{|B^\lambda|}\right)^{1 + \frac{1}{2n}}. \end{equation}
For any $\rho > 0$ let $F_1 = F \cap B_\rho(0)$ and $F_2 = F \setminus B_\rho(0)$.
Note that for every $\epsilon > 0$ there exists $\bar m$ sufficiently small such that, if $\rho \ge \frac{m^{\frac{1}{n}}}{|B^\lambda|^{\frac{1}{n}}} R_\lambda =: \rho_m$, with $m < \bar m$, then
\begin{equation}\label{4.4:km} |F_2| \le \epsilon |F_1|.
\end{equation} 
Indeed, since $B^\lambda(|B^\lambda|) \subset B_{R_\lambda}(0)$, we get $B^\lambda(m) \subset B_\rho(0)$.
Moreover, by~\eqref{5.3:km} and for sufficiently small $\bar m$ we estimate
\begin{equation*} \begin{split} |F_1| = |F \cap B_\rho(0)| &\ge |F \cap B^\lambda(m)| \\ &= |F| - |F \setminus B^\lambda(m)| \\ &\ge |B^\lambda| \left(\frac{m}{|B^\lambda|}\right) - c(n, \lambda, g, G) \left(\frac{m}{|B^\lambda|}\right)^{1 + \frac{1}{2n}} \\ & 
\ge \frac{|B^\lambda|}{2} \left(\frac{m}{|B^\lambda|}\right).
\end{split} \end{equation*}
and \begin{equation*} |F \setminus B_\rho(0)| \le |F \setminus B^\lambda(m)| \le c(n, \lambda, g, G) \left(\frac{m}{|B^\lambda|}\right)^{1 + \frac{1}{2n}} \le c(n, \lambda, g, G) \left(\frac{m}{|B^\lambda|}\right)^{\frac{1}{2n}}|F_1| \le \epsilon |F_1|.
\end{equation*}
Let us define the monotonically decreasing function $U(\rho) = |F \setminus B_\rho(0)|$.\\
We now distinguish two cases.
Let us firstly prove~\eqref{5.1:km} when we assume that \begin{equation}\label{5.4:km} \Sigma := P_\lambda(F_1) + P_\lambda(F_2) - P_\lambda(F) > \frac{1}{2} \ff^\lambda(F_2) \qquad \forall \rho \in \left(\rho_m, \frac{R_\lambda}{2}\right) \end{equation}
By~\eqref{5.3:km} we have \begin{equation*} U(\rho_m) = |F \setminus B^\lambda(m)| \le c(n, \lambda, g, G) \, m^{1 + \frac{1}{2n}} \le c(n, \lambda, g, G) \rho_m^{n + \frac{1}{2}}  \end{equation*}
Furthermore, by~\eqref{isoperimetric:inequality} and~\eqref{5.4:km} we have \begin{equation*}\label{5.5:km} \begin{split} -2\frac{\d U(\rho)}{\d\rho} & = \Sigma > \frac{1}{2}\ff^\lambda(F_2) \ge \frac{1}{2} P_\lambda(F \setminus B_\rho(0)) \ge c(n, \lambda) U^{\frac{n - 1}{n}}(\rho). \end{split} \end{equation*} 
In particular we have \begin{equation*} \begin{cases} \quad \frac{\d U(\rho)}{\d\rho} \le - c(n, \lambda) U^{\frac{n - 1}{n}}(\rho) \qquad & {\rm for}\,{\rm a.e.} \; \rho \in \left(\rho_m, \frac{R_\lambda}{2}\right) \\ \quad U(\rho_m) \le c(n, \lambda, g, G) \rho_m^{n + \frac{1}{2}}. \qquad & \end{cases} \end{equation*}
By ODE comparison we deduce that, if $\bar m < 1$, \begin{equation*} \begin{split} U(\rho)^{\frac{1}{n}} & \le U(\rho_m)^{\frac{1}{n}} - c(n, \lambda) \, (\rho - \rho_m) \\ & \le c(n, \lambda, g, G)\, \rho_m^{1 + \frac{1}{2n}} - c(n, \lambda) \rho + c(n, \lambda) \rho_m \\ & = c(n, \lambda, g, G) \, m^{\frac{1}{n} + \frac{1}{2n^2}} - c(n, \lambda)\, \rho + c(n, \lambda)\, m^{\frac{1}{n}} \\ & \le c(n, \lambda, g, G) \, m^{\frac{1}{n}} - c(n, \lambda)\, \rho + c(n, \lambda)\, m^{\frac{1}{n}} \\ & = c(n, \lambda, g, G) \, m^{\frac{1}{n}} - c(n, \lambda)\, \rho. \end{split} \end{equation*}
For $\bar m$ sufficiently small, it follows that $U(\rho) = 0$ for $\rho \ge \frac{R_\lambda}{2}$, and we obtain~\eqref{5.1:km} with $L = F$. 
\\ Let us prove~\eqref{5.1:km} assuming that \begin{equation}\label{4.3:km} \Sigma \le \frac{1}{2} \ff^\lambda(F_2) \end{equation} holds for some $\rho_0 \in \left(\rho_m, \frac{R_\lambda}{2}\right)$.
Let $m_1 := |F_1|$, $m_2 := |F_2|$ and $\gamma := \frac{m_2}{m_1} \le \epsilon$, with $\epsilon$ that will be chosen suitably small later.
Let us also denote $\tilde F = l \, F_1$, with $l := (1 + \gamma)^{\frac{1}{n}}$.
In particular $|\tilde F| = m$ and, if $\epsilon$ is sufficiently small, \begin{equation*} \tilde F = (1 + \gamma)^{\frac{1}{n}} F_1 = (1 + \gamma)^{\frac{1}{n}} \left(F \cap B_{\rho_0}(0)\right) \subset B_{\rho_0\sqrt[n]{1 + \gamma}}(0) \subset B_{R_\lambda}(0) \subset \bar Q_1. \end{equation*}
By \cref{lemma3.1:np} \begin{equation}\label{4.6:km} \begin{split} \ff^\lambda(\tilde F) = \ff^\lambda(l F_1) & \le l^{2n + q} \ff^\lambda(F_1) \\ & = \ff^\lambda(F_1) + \left(l^{2n + q} - 1\right) \ff^\lambda(F_1). \end{split} \end{equation}
Choosing $\epsilon \le 1$, we have $1 \le l \le 2^{\frac{1}{n}}$, and by Taylor's formula we obtain $l^{2n + q} - 1 = (1+\gamma)^{2+q/n}-1 \le \gamma K$ for some $K > 0$ independent of $\gamma$ and for $\epsilon$ sufficiently small.
By~\eqref{4.6:km} we arrive at \begin{equation*} \ff^\lambda(\tilde F) - \ff^\lambda(F_1) \le \gamma K \ff^\lambda(F_1). \end{equation*}
By the definition of $\Sigma$ and since $\rr(F_1) + \rr(F_2) \le \rr(F)$ \begin{equation}\label{4.7:km} \begin{split} \ff^\lambda(\tilde F) - \ff^\lambda(F) & \le \rr(F_1) + \g(F_1) + \rr(F_2) + \g(F_2) - \rr(F) - \g(F) + \Sigma - \ff^\lambda(F_2) + \gamma K \ff^\lambda(F_1) \\ & \le  - \frac{1}{2} \ff^\lambda(F_2) + \gamma K \ff^\lambda(F_1). \end{split} \end{equation}
By positivity of $\rr$ and $\g$ and the isoperimetric inequality~\eqref{isoperimetric:inequality}, we have $\ff^\lambda(F_2) > P_\lambda(F_2) \ge c(n, \lambda) m_2^{\frac{n - 1}{n}}$.
By~\eqref{4.3:km} we obtain \begin{equation}\label{zz:eqStimaEnergia} \begin{split} \ff^\lambda(F) - \ff^\lambda(F_1) & = P_\lambda(F) + \rr(F) + \g(F) - P_\lambda(F_1) - \rr(F_1) - \g(F_1) - P_\lambda(F_2) + P_\lambda(F_2) \\ & \ge - \frac{1}{2} \ff^\lambda(F_2) + \rr(F) + \g(F) - \rr(F_1) - \g(F_1) + P_\lambda(F_2) \\ & = - \frac{1}{2} P_\lambda(F_2) - \frac{1}{2} \rr(F_2) - \frac{1}{2} \g(F_2) + \rr(F) + \g(F) - \rr(F_1) - \g(F_1) + P_\lambda(F_2) \ge 0, \end{split} \end{equation} that is $\ff^\lambda(F_1) \le \ff^\lambda(F)$.
By~\eqref{4.5:km}, since $\gamma m \le 2 m_2$ and $\gamma \le \epsilon$,~\eqref{4.7:km} turns into \begin{equation*}\label{4.8:km} \begin{split} \ff^\lambda(\tilde F) - \ff^\lambda(F) & \le - c(n, \lambda) m_2^{\frac{n - 1}{n}} + \gamma K \ff^\lambda(F) \\ & \le - c(n, \lambda) m_2^{\frac{n - 1}{n}} + C(n, \lambda, g, G, q) \max\left\{m_2, \epsilon^{\frac{1}{n}} m_2^{\frac{n - 1}{n}}\right\}. \end{split} \end{equation*}
Since $m_2 \le c(n, \lambda) \epsilon$ by~\eqref{4.4:km}, for $\epsilon$ sufficiently small~\eqref{5.1:km} follows with $L = \tilde F$. \end{proof}

\begin{remark}\label{remark:quantitative} Arguing exactly as in~\eqref{4.5:km}, \eqref{5.2:km} and~\eqref{5.2:kmbis}, we easily deduce that if $E$ is a minimizer of $\ff^\lambda$ in $\mathcal{A}_m$ with $m$ sufficiently small then \begin{equation*} \alpha_\lambda(E) \le c(n, \lambda, g, G) m^{\frac{1}{2n}}. \end{equation*} \end{remark}

Now we are ready to prove \cref{simple:existence}.
\begin{proof}[Proof of \cref{simple:existence}] By \cref{lemma5.1:km} there exists a minimizing sequence with uniformly bounded sets. 
The lower semicontinuity of $P_\lambda$ \cite[Lemma 3.7]{PascalePozzettaQuantitative} and the continuity of $\rr$ and $\g$ under strong $L^1$ convergence (which holds by dominated convergence theorem and uniformly boundedness of the minimizing sequence) allow to conclude the proof. \end{proof}

The following proposition states that, if the nonlocal kernel $g$ and the gravitational term $G$ are coercive, we have existence of minimizers for \emph{all} values of $m$.
\begin{proposition}\label{prop:existence} Let $g$ be a $\rr$-admissible coercive function and let $G$ be a $\g$-admissible coercive function.
Then, for every $m > 0$, there exists a minimizer of $\ff^\lambda$ in the class \begin{equation*} \mathcal{A}_m := \{\Omega \subset \R^n \setminus H \,{\rm measurable} \st |\Omega| = m\}. \end{equation*} \end{proposition}
\begin{proof} Let us consider a minimizing sequence $\left\{E_i\right\}_{i \in \N}$ of $\ff^\lambda$ in $\mathcal{A}_m$.
In particular \begin{equation*} \sup_{i \in \N}\ff^\lambda(E_i) = C < + \infty. \end{equation*} 
By \cite[Corollary 2.5]{PascalePozzettaQuantitative} it holds \begin{equation*} P(E_i) \le \frac{2}{1 - \lambda}P_\lambda(E_i) \le \frac{2}{1 - \lambda}\ff^\lambda(E_i) \le C, \qquad \forall i \in \N. \end{equation*}
By \cite[Lemma 2.10]{AntonelliNardulliPozzetta} and \cite[Corollary 3.25]{LeonardiRitoreVernadakis} we get the existence of a constant $\bar c = \bar c\left(m, n, \sup_i P(E_i)\right)$ $> 0$ such that for every $i \in \N$ there exists $x_i \in \left\{x_n > 0\right\}$ with \begin{equation*} \left|E_i \cap B_1(x_i)\right| \ge \bar c. \end{equation*}
By the coercivity of $G$ we have that $\left(x_i\right)_n$ is uniformly bounded.
Indeed, if by contradiction $\left(x_i\right)_n \to + \infty$, then \begin{equation*} C \ge \int_{E_i} G(x_n) \de x \ge \int_{E_i \cap B_1(x_i)} G(x_n) \de x \ge \bar c \left(\inf_{B_1(x_i)} G(x_n)\right) \to + \infty. \end{equation*}
By the lower semicontinuity of $\ff^\lambda$ under strong $L^1$ convergence, the existence of a minimizer follows if we show that for every $\epsilon > 0$ and $i \in \N$ there exists $R > 0$ with \begin{equation}\label{existence:coercive} \sup_{i \in \N} \left|E_i \setminus B_R(x_i)\right| < \epsilon. \end{equation}
In order to prove~\eqref{existence:coercive}, note that \begin{equation*} \begin{split} C \ge \rr(E_i) & \ge \int_{E_i}\int_{E_i \setminus B_R(x)} g(y - x) \de y \de x \ge \left(\inf_{|x| > R} g(x)\right) \int_{E_i} \left|E_i \setminus B_R(x)\right| \de x \\ & \ge \left(\inf_{|x| > R} g(x)\right) \int_{E_i \cap B_1(x_i)} \left|E_i \setminus B_R(x)\right| \de x \ge \bar c \left(\inf_{|x| > R} g(x)\right) \left|E_i \setminus B_{R + 1}(x_i)\right| \end{split} \end{equation*} and \begin{equation*} \left|E_i \setminus B_{R + 1}(x_i)\right| \le \frac{C}{\bar c} \left(\inf_{|x| > R} g(x)\right)^{-1} \to 0 \qquad {\rm as} \quad R \to + \infty. \end{equation*} \end{proof}

\subsection{Boundedness and indecomposability of minimizers}\label{subsec:3.2}

In this section we will prove two qualitative properties of volume constrained minimizers of $\ff^\lambda$, namely boundedness and indecomposability.
We begin with the following

\begin{theorem}\label{teo:limitato} Let $g$ be $\rr$-admissible, infinitesimal and symmetric and let $G$ be $\g$-admissible.
Let $E \subset \R^n \setminus H$ be a minimizer of $\ff^\lambda$ with $|E| = m$, $m > 0$.
Then $E$ is essentially bounded. \end{theorem}

\begin{remark} We remark that \cref{teo:limitato} proves boundedness of minimizers without requiring growing properties of the Riesz-type kernel, but only infinitesimality and symmetry. \end{remark}

Before giving the proof, we recall the definition and some properties of the so-called $(K, r_0)$-quasiminimal sets.
\begin{definition}\label{quasiminimal:set} Let $E \subset \R^n \setminus H$ be a set of finite perimeter with finite measure, and let $K \ge 1$, $r_0 > 0$.
We say that $E$ is a $(K, r_0)$-quasiminimal set (relatively in $\R^n \setminus H$) if \begin{equation*} P(E, \R^n \setminus H) \le K P(F, \R^n \setminus H), \end{equation*} for any $F \subset \R^n \setminus H$ such that $E \Delta F \subset\subset B_r(x)$, for some ball $B_r(x) \subset \R^n$ with $r \le r_0$ and $x \in \{x_n \ge 0\}$. \end{definition}
\begin{theorem}\label{density:quasiminimal} Let $E \subset \R^n \setminus H$ be a $(K, r_0)$-quasiminimal set, for some $K \ge 1$, $r_0 > 0$. 
Then there exist $c = c(n, K, r_0) \in \left(0, \frac{1}{2}\right]$ and $r'_0 = r'_0(n, K, r_0) \in (0, r_0]$ such that \begin{equation*} c \le \frac{|E \cap B_r(x)|}{|B_r(x) \setminus H|} \le 1 - c \qquad \forall x \in \overline{\partial E \setminus H}, \; \forall r \in (0, r'_0]. \end{equation*}
In particular the set $E^{(1)}$ of points of density $1$ for $E$ is an open representative for $E$. \end{theorem}
The proof of \cref{density:quasiminimal} follows, for instance, by repeatedly applying \cite[Theorem 4.2]{ShanmugalingamQuasiminimizers} with $X = \{x_n \ge 0\}$ in domains $\Omega = X \cap B_{r_0}(x)$ for $x \in X$, in the notation of \cite[Theorem 4.2]{ShanmugalingamQuasiminimizers}. 
Observe also that in \cite{ShanmugalingamQuasiminimizers}, the perimeter functional coincides with the relative perimeter in $\R^n \setminus H$, hence the definition of quasiminimal set in \cite[Definition 3.1]{ShanmugalingamQuasiminimizers} coincides with \cref{quasiminimal:set}.
Alternatively, \cref{density:quasiminimal} follows also by adapting the classical argument in the proof of \cite[Theorem 21.11]{MaggiBook} working with $(K, r_0)$-quasiminimal sets instead of $(\Lambda, r_0)$-minimizers. \par
\vspace{2mm}

The aim of the following lemmas is to prove that minimizers of $\ff^\lambda$ are $(K, r_0)$-quasiminimal sets, in order to apply \cref{density:quasiminimal}.

\begin{lemma}\label{vertical:boundedness}
Let $g$ be $\rr$-admissible, infinitesimal and symmetric and let $G$ be $\g$-admissible.
Let $E \subset \R^n \setminus H$ be a minimizer of $\ff^\lambda$ with $|E| = m$, $m > 0$.
Then there exists $\bar x_n > 0$, depending on $n$, $g$, $G$, $E$, such that \begin{equation}\label{xn:definition} |E \cap \{x_n > \bar x_n\}| = 0. \end{equation} \end{lemma}

\begin{proof} Let us define, for every $t > 0$, \begin{equation*} E_t := E \cap \{x_n \le t\}, \qquad V(t):= \left|E \cap \{x_n > t\}\right|. \end{equation*}
Fix $x_0 \in \partial^* E$ such that $ x_0 \in \partial^* E \cap \{0 < x_n < t\}$ and $r_0 > 0$ such that $B_{r_0}(x_0) \subset\subset \{0 < x_n < t\}$ for any $t$ large enough.
By \cite[Lemma 17.21]{MaggiBook} there exist $\sigma_0$, $c_0 \in (0, \infty)$, depending on $E$, $x_0$ and $r_0$, such that for every $\sigma \in (- \sigma_0, \sigma_0)$ we can find a set of finite perimeter $F$, given by a suitable local variation of $E$, such that \begin{equation}\label{piccola:deformazione} F \Delta E \subset \subset B_{r_0}(x_0) \qquad |F| = |E| + \sigma, \qquad |P(F, B_{r_0}(x_0)) - P(E, B_{r_0}(x_0))| \le c_0|\sigma|. \end{equation}
Now consider $t_0 = t_0(E) > 0$ large enough such that $V(t_0) < \sigma_0$, and set $\sigma = V(t)$ for $t > t_0$.
Then there exists $\tilde F$ such that~\eqref{piccola:deformazione} holds.
Define also $\tilde E_{t} := \tilde F \cap \{0 < x_n \le t\}$, so that \begin{equation*} \begin{split} |\tilde E_{t}| & = |\tilde F| - |\tilde F \cap \{x_n > t\}| = |\tilde F| -V(t) = |E| + \sigma - \sigma = |E|. \end{split} \end{equation*}
Moreover by \cite[Lemma 17.9, Lemma 17.21]{MaggiBook} and properties of local variations, we get
\begin{equation*} \begin{split} & |\tilde E_{t} \Delta E_{t}| = |\tilde F \Delta E| \le c(E) \,  | |\tilde F| - |E| | = c(E) V(t),\\ & |P(\tilde E_{t}, B_{r_0}(x_0)) - P(E, B_{r_0}(x_0))| \le c(E)\, V(t). \end{split} \end{equation*}
By the minimality of $E$ \begin{equation*} P_\lambda(E) + \rr(E) + \g(E) \le P_\lambda(\tilde E_{t}) + \rr(\tilde E_{t}) + \g(\tilde E_{t}). \end{equation*}
Since $\hh^{n - 1}(\partial^*E \cap \partial H) = \hh^{n - 1}(\partial^*\tilde E_{t} \cap \partial H)$, we get \begin{equation*} \begin{split} P(E, \R^n \setminus H) + \rr(E) + \g(E) & \le P(\tilde E_{t}, \R^n \setminus H) + \rr(\tilde E_{t}) + \g(\tilde E_{t}) \\ & \le P(E_{t}, \R^n \setminus (H \cup B_{r_0}(x_0))) + P(E, B_{r_0}(x_0)) + c(E)\, V(t) + \rr(\tilde E_{t}) + \g(\tilde E_{t}) \\ & = P(E, \{x_n < t\}) + |V'(t)| + \rr(\tilde E_{t}) + \g(\tilde E_{t}) + c(E) V(t) \\ & = P(E, \R^n \setminus H) - P(E, \{x_n > t\}) + |V'(t)| + \rr(\tilde E_{t}) + \g(\tilde E_{t}) + c(E)\,V(t). \end{split} \end{equation*}
Then \begin{equation*} P(E, \{x_n > t\}) \le |V'(t)| + \rr(\tilde E_{t}) - \rr(E) + \g(\tilde E_{t}) - \g(E) + c(E)\, V(t) \end{equation*}
By Fubini theorem and symmetry of $g$
\begin{equation*} \begin{split} \rr(\tilde E_{t}) - \rr(E) & = \int_{\tilde E_{t} \setminus E}\int_{\tilde E_{t}}g(y - x) \de y \de x + \int_{\tilde E_{t} \cap E}\int_{\tilde E_{t}\setminus E} g(y - x) \de y \de x \\ & \quad - \int_{E \setminus \tilde E_{t}} \int_E g(y - x) \de y \de x - \int_{E \cap \tilde E_{t}} \int_{E \setminus \tilde E_{t}} g(y - x) \de y \de x \\ & = \int_{\tilde E_{t} \setminus E}\int_{\tilde E_{t}}g(y - x) \de y \de x + \int_{\tilde E_{t} \setminus E}\int_{\tilde E_{t}\cap E} g(x - y) \de y \de x \\ & \quad - \int_{E \setminus \tilde E_{t}} \int_E g(y - x) \de y \de x - \int_{E \setminus \tilde E_{t}} \int_{E \cap \tilde E_{t}} g(x - y) \de y \de x \end{split} \end{equation*} \begin{equation*} \begin{split} & = \int_{\tilde E_{t} \setminus E}\int_{\tilde E_{t}}g(y - x) \de y \de x + \int_{\tilde E_{t} \setminus E}\int_{\tilde E_{t}\cap E} g(y - x) \de y \de x \\ & \quad - \int_{E \setminus \tilde E_{t}} \int_E g(y - x) \de y \de x - \int_{E \setminus \tilde E_{t}} \int_{E \cap \tilde E_{t}} g(y - x) \de y \de x \\ & \le \int_{\tilde E_{t} \Delta E} \left(\int_{\tilde E_{t}} g(y - x) \de y + \int_E g(y - x) \de y\right) \de x \end{split} \end{equation*} 
Since $g$ is infinitesimal there exists $R_g > 0$ such that \begin{equation*} g(x) < 1 \qquad \forall x \st |x| > R_g. \end{equation*}
Then \begin{equation*} \begin{split} \rr(\tilde E_{t}) - \rr(E) & \le \int_{\tilde E_{t} \Delta E} \left(2 \int_{B_{R_g}(0)}g(z) \de z + \left|\tilde E_{t}\setminus B_{R_g(0)}\right| +  \left|E \setminus B_{R_g(0)}\right|\right) \de x \\ & \le 2 \int_{\tilde E_{t} \Delta E} \left(\int_{B_{R_g}(0)}g(z) \de z + |E|\right) \de x \\ & \le c(g, m) (|\tilde E_{t} \Delta E_{t}| + |E_{t} \Delta E|) \le c(g, E)V(t). \end{split} \end{equation*}
By \cref{global:growthbis} $\left(\sup_{B_{r_0}(x_0)} G\right) < \infty$ and
\begin{equation*} \begin{split} \g(\tilde E_{t}) - \g(E) & = \int_{\tilde E_{t} \setminus E} G \de x - \int_{E \setminus \tilde E_{t}} G \de x \le \int_{\tilde E_{t} \setminus E} G \de x \le \left(\sup_{B_{r_0}(x_0)} G\right) \,\,|\tilde E_t \setminus E|. \end{split} \end{equation*}
Therefore, for almost every $t$ sufficiently large,
\begin{equation}\label{eq:verticalbound} P(E, \{x_n > t\}) \le |V'(t)| + c(n, g, G, E)\,V(t). \end{equation}
Finally, if $c_{iso} = c_{iso}(n)$ is the constant in the classical isoperimetric inequality and $t$ is large enough,~\eqref{eq:verticalbound} yields \begin{equation*} \begin{split} c_{iso} V(t)^{\frac{n - 1}{n}} = c_{iso} |E \setminus E_t|^{\frac{n - 1}{n}} & \le P(E \setminus E_t) = \hh^{n - 1}(\partial^*E_t \cap \{x_n = t\}) + P(E, \{x_n > t\}) \\ & \le 2 |V'(t)| + c(n, g, G, E)\,V(t) < 2 |V'(t)| +  \frac{c_{iso}}{2} V(t)^{\frac{n - 1}{n}} \end{split} \end{equation*} and \begin{equation*} - V'(t) \ge c V(t)^{\frac{n - 1}{n}}. \end{equation*}
Therefore ODE comparison implies that $V(t)$ vanishes at some $t = \bar x_n < + \infty$. \end{proof}

\begin{lemma}\label{lemma:quasiminimal} Let $g$ be $\rr$-admissible, infinitesimal and symmetric and let $G$ be $\g$-admissible. 
Let $E \subset \R^n \setminus H$ be a minimizer of $\ff^\lambda$ with $|E| = m$, $m > 0$.
Then $E$ is a $(K, r_0)$-quasiminimal set, for suitable $K \ge 1$ and $r_0 > 0$, depending on $n$, $\lambda$, $g$, $G$, $E$. \end{lemma}
\begin{proof} Let us consider $x_1$, $x_2 \in \partial^*E \setminus H$ and $t_0 > 0$ such that we have $B_{t_0}(x_1) \cap B_{t_0}(x_2) = \emptyset$ and $B_{t_0}(x_1) \cup B_{t_0}(x_2) \subset \subset \R^n \setminus H$.
By applying \cite[Lemma 17.21]{MaggiBook} we find two positive constants $\sigma_0$ and $c_0$, depending on $E$, such that, given $|\sigma| < \sigma_0$, there exist two sets of finite perimeter $F_1$ and $F_2$ with \begin{equation}\label{maggi:21.4} E \Delta F_k \subset \subset B_{t_0}(x_k), \qquad |F_k| = |E| + \sigma, \qquad \left|P(E, B_{t_0}(x_k)) - P(F_k, B_{t_0}(x_k))\right| \le c_0 |\sigma|, \quad k \in \{1, 2\}. \end{equation}
Let $r_0 = r_0(n, \lambda, g, G, E) > 0$ to be determined later.
At the moment assume that 
\begin{equation*} r_0 < \min \left\{\frac{t_0}{2}, \frac{\sigma_0^{\frac{1}{n}}}{\omega_n}, \frac{|x_1 - x_2| - 2t_0}{2}\right\}. \end{equation*}
In particular, if a ball of radius $r_0$ intersects $B_{t_0}(x_1)$ (resp. $B_{t_0}(x_2)$), then it is disjoint from $B_{t_0}(x_2)$ (resp. from $B_{t_0}(x_1)$).
Let $F$ be such that $E \Delta F \subset\subset B_r(x) \cap (\R^n \setminus H)$, where $r < r_0$.
Then, by the definition of $r_0$, \begin{equation*} \left||E| - |F|\right| \le |E \Delta F| \le \omega_n r^n < \omega_nr_0^n \le \sigma_0 \end{equation*} and we can compensate for the volume deficit $||E| - |F||$ between $E$ and $F$ by modifying $F$ inside either $B_{t_0}(x_1)$ or $B_{t_0}(x_2)$.
Precisely, by the definition of $r_0$, we may assume without loss of generality that $B_r(x)$ does not intersect $B_{t_0}(x_1)$, set $\sigma = |E| - |F|$, and consider $F_1$ verifying~\eqref{maggi:21.4}, so that \begin{equation}\label{maggi:21.5} E \Delta F_1 \subset\subset B_{t_0}(x_1), \qquad E \Delta F \subset\subset B_r(x) \cap (\R^n \setminus H) \subset\subset \R^n \setminus \left(H \cup \overline{B_{t_0}(x_1)}\right). \end{equation}
By~\eqref{maggi:21.4} $\sigma = |F_1| - |E|$ and, if we define \begin{equation*} \tilde F = (F \cap B_r(x)) \cup (F_1 \cap B_{t_0}(x_1)) \cup (E \setminus (B_r(x) \cup B_{t_0}(x_1))), \end{equation*} then $|\tilde F| = |E|$ and $\tilde F \Delta E \subset \subset \{x_n > 0\}$.
By the minimality of $E$ \begin{equation}\label{quasiminimal:1} \begin{split} (1 - |\lambda|) P(E, \R^n \setminus H) \le P_\lambda(E) & \le P_\lambda(\tilde F) + \rr(\tilde F) - \rr(E) + \g(\tilde F) - \g(E). \end{split} \end{equation} 
By~\eqref{maggi:21.4} and~\eqref{maggi:21.5} we get \begin{equation}\label{quasiminimal:2} \begin{split} P_\lambda(\tilde F) & \le (1 + |\lambda|) P(\tilde F, \R^n \setminus H) \\ & \le (1 + |\lambda|)\left[P(\tilde F, \R^n \setminus (H \cup \overline{B_{t_0}(x_1)})) + P(\tilde F, B_{t_0}(x_1)) + P(\tilde F, \partial B_{t_0}(x_1))\right] \\ & = (1 + |\lambda|) \left[P(F, \R^n \setminus (H \cup \overline{B_{t_0}(x_1)})) + P(F_1, B_{t_0}(x_1)) + P(F, \partial B_{t_0}(x_1))\right] \\ & \le (1 + |\lambda|) \left[P(F, \R^n \setminus (H \cup B_{t_0}(x_1))) + P(E, B_{t_0}(x_1)) + c_0(E) |\sigma|\right] \\ & \le (1 + |\lambda|) \, P(F, \R^n \setminus H) + c_0(\lambda, E) |F \Delta E|. \end{split} \end{equation}
As in the proof of \cref{vertical:boundedness} one estimates
\begin{equation}\label{quasiminimal:3} \rr(\tilde F) - \rr(E) \le c(g, E)\,|F \Delta E|, \end{equation} 
and by \cref{global:growthbis},~\eqref{xn:definition} and if $t_0 < 1$ 
\begin{equation}\label{quasiminimal:4} \begin{split} \g(\tilde F) - \g(E) & \le \left(\sup_{\tilde F \setminus E}G\right) |\tilde F \setminus E| \le \left(\sup_{\tilde F \setminus E}G\right) \left(|\tilde F \Delta F| +|F \Delta E|\right) \\ & \le c(G, E) \,(\bar x_n + 1)^n |F \Delta E| = c(n, g, G, E) \, |F \Delta E|. \end{split} \end{equation}
However, by the relative isoperimetric inequality \cite{ChoeGhomiRitoreInequality, FuscoMorini} and~\eqref{maggi:21.5} \begin{equation}\label{quasiminimal:5} \begin{split} |F \Delta E| = |F \Delta E|^{\frac{1}{n}} |F \Delta E|^{\frac{n - 1}{n}} & \le c(n) |F \Delta E|^{\frac{1}{n}} P(F \Delta E, \R^n \setminus H) \\ & \le c(n) |F \Delta E|^{\frac{1}{n}} (P(F, \R^n \setminus H) + P(E, \R^n \setminus H)) \\ & \le c(n) \,r_0\, (P(F, \R^n \setminus H) + P(E, \R^n \setminus H)). \end{split} \end{equation}
Putting together~\eqref{quasiminimal:1}-\eqref{quasiminimal:5} we obtain \begin{equation*} \left((1 - |\lambda|) - c(n, \lambda, g, G, E) \, r_0\right) P(E, \R^n \setminus H) \le \left((1 + |\lambda|) + c(n, \lambda, g, G, E) \, r_0\right) P(F, \R^n \setminus H) \end{equation*}
If $r_0$ is sufficiently small, we conclude the proof. \end{proof}
By \cref{density:quasiminimal} and \cref{lemma:quasiminimal}, from now on we can identify any minimizer $E$ of $\ff^\lambda$, with $|E| = m$, $m > 0$, $g$ $\rr$-admissible infinitesimal symmetric function and $G$ $\g$-admissible function, with the open set $E^{(1)}$ of points of density $1$ for $E$. \par
\vspace{2mm}

Now we are ready to prove \cref{teo:limitato}.

\begin{proof}[Proof of \cref{teo:limitato}] By \cref{lemma:quasiminimal} and \cref{density:quasiminimal} there exist $r > 0$ and $c > 0$ such that for every $x \in \overline{\partial E\setminus H}$ we have $|E \cap B_r(x)| \ge c r^n$.
If $E$ were not bounded, one would easily get $|E| = \infty$. \end{proof}

Now we prove indecomposability of minimizers.

\begin{theorem}\label{teo:connesso}
Let $g$ be $\rr$-admissible and infinitesimal and let $G$ be $\g$-admissible.
Let $E \subset \R^n \setminus H$ be a minimizer of $\ff^\lambda$ with $|E| = m$, $m > 0$.
Then $E$ is indecomposable. \end{theorem}

\begin{proof}
We argue by contradiction.
Assume that there exist two sets of finite perimeter $E_1$ and $E_2$ such that $|E_1 \cap E_2| = 0$, $E = E_1 \cup E_2$ and $P(E) = P(E_1) + P(E_2)$.
If $R > 0$ is sufficiently large, letting $e_1=(1,0,\ldots,0)$ and defining $E_R := E_1 \cup (E_2 + e_1 R)$, we have $|E_R| = m$, $P_\lambda(E_R) = P_\lambda(E)$ and $\g(E_R) = \g(E)$.
At the same time, the nonlocal energy decreases, precisely \begin{equation*} \liminf_{R \to \infty} \left(\int_{E_1}\int_{E_2 + e_1 R} g(y - x) \de y \de x + \int_{E_2 + e_1 R}\int_{E_1} g(y - x) \de y \de x\right) = 0, \end{equation*} 
and \begin{equation*}\label{km:4.2} \begin{split} \liminf_{R \to \infty} \ff^\lambda(E_R) & = P_\lambda(E) + \rr(E_1) + \rr(E_2) + \g(E) \\ & < P_\lambda(E) + \rr(E_1) + \rr(E_2) + \int_{E_1} \int _{E_2}g(y - x) \de y \de x + \int_{E_2} \int _{E_1}g(y - x) \de y \de x + \g(E) \\ & = \ff^\lambda(E). \end{split} \end{equation*}
Therefore, if $R$ is sufficiently large, we obtain $\ff^\lambda(E_R) < \ff^\lambda(E)$, in contradiction with the minimizing property of $E$. \end{proof}

\section{Nonexistence of minimizers for large masses}\label{sec:4}

The goal of this Section is to prove \cref{theorem:2}. 
We begin by proving some preparatory lemmas.
Let us start with a non-optimality criterion.

\begin{lemma}\label{lemma4.2:km}
Let $g$ be $\rr$-admissible, $q$-growing and infinitesimal and let $G$ be $\g$-admissible. 
There exists $\epsilon>0$, depending on $n$, $\lambda$, $g$, $G$ and $q$, such that the following holds. 
Let $F \subset \R^n \setminus H$ be a set of finite perimeter and assume there exist two sets of finite perimeter $F_1$, $F_2 \subset F$ such that $|F_1|,$ $|F_2|>0$, $|F_1 \cap F_2|=0$, $|F\setminus (F_1 \cup F_2)|=0$ and 
\begin{equation}\label{4.3:km1} \Sigma := P_\lambda(F_1) + P_\lambda(F_2) - P_\lambda(F) \le \frac{1}{2} \ff^\lambda(F_2). \end{equation}
Then, if \begin{equation}\label{4.4:km1} |F_2| \le \epsilon \min \{1, |F_1|\}, \end{equation} there exists a set $G \subset \R^n \setminus H$ with $|G| = |F|$ and $\ff^\lambda(G) < \ff^\lambda(F)$. \end{lemma}

\begin{proof}
Let us denote $m := |F|$, $m_1 := |F_1|$, $m_2 := |F_2|$ and $\gamma := \frac{m_2}{m_1} \le \epsilon$.
Let us define the sets $\tilde F$ and $\hat F$ in the following way: $\tilde F$ is given by $\tilde F = l \, F_1$, with $l := \sqrt[n]{1 + \gamma}$, so that $|\tilde F| = |F|$, and $\hat F$ is given by a collection of $N \ge 1$ spherical caps $\left\{B^\lambda(v, x_i)\right\}_{1 \le i \le N}$ of equal volume $v$ and with centers located at $x_i = i R e_1$, $i = 1, \dots, N$, with $R$ large enough so that the $B^\lambda(v, x_i)$, $1 \le i \le N$ are pairwise disjoint.
The number $N$ is the smallest integer for which the volume of each spherical cap does not exceed $\min \left\{1, |B^\lambda|\right\}$. Hence $Nv = m$ and $N = \left\lceil \frac{m}{\min\{1, |B^\lambda|\}} \right\rceil$.
If there exists $R > 0$ such that, for the corresponding $\hat F$, one has $\ff^\lambda(\hat F) < \ff^\lambda(F)$, then the proof is concluded with $G = \hat F$. So we can assume that for any $R > 0$ there holds $\ff^\lambda(\hat F) \ge \ff^\lambda(F)$. Hence for $R$ large enough we claim that
\begin{equation}\label{4.5:km1}
\ff^\lambda(F) \le \ff^\lambda(\hat F) \le c(n, \lambda, g, G) \max\left\{m, m^{\frac{n - 1}{n}}\right\}. \end{equation}
Indeed, if $m \ge 1$, estimate \eqref{4.5:km1} follows by the same computations done in the proof of \cref{theorem3.4:km}.
If instead $m < 1$, by \cite[Lemma 3.1]{PascalePozzettaQuantitative} we find
\begin{equation*} \begin{split} P_\lambda(E) = P_\lambda\left(\cup_{i=1}^N B^\lambda(v, x_i)\right) & = \sum_{i=1}^N P_\lambda(B^\lambda(v, x_i)) = c(n, \lambda) N v^{\frac{n - 1}{n}} \le c(n, \lambda) \left(\frac{m}{\min\{1, |B^\lambda|\}} + 1\right) v^{\frac{n - 1}{n}} \\ & \le c(n, \lambda) ( m \, v^{\frac{n - 1}{n}} + v^{\frac{n - 1}{n}}) \le c(n, \lambda) ( m + m^{\frac{n - 1}{n}} ) = c(n, \lambda) m^{\frac{n - 1}{n}}.
\end{split} \end{equation*}
Arguing as in \cref{theorem3.4:km}, if $R$ is so large that $g(x - y) < \frac{1}{N}$ for every $x \in B^\lambda(v, x_j)$, $y \in B^\lambda(v, x_k)$ with $j \neq k$, then \begin{equation*} \begin{split} \rr (E) & \le c(n, \lambda, g, G) m \le c(n, \lambda, g, G) m^{\frac{n - 1}{n}} \end{split} \end{equation*}
and
\begin{equation*}
\begin{split} \ff^\lambda(E) & = P_\lambda(E) + \rr(E) + \g (E) \le c(n, \lambda) |E|^{\frac{n - 1}{n}} + c(n, \lambda, g, G) |E|^{\frac{n - 1}{n}} + \sum_{i = 1}^N \g(B^\lambda(v, x_i)) \le c(n, \lambda, g, G) m^{\frac{n - 1}{n}}, \end{split} \end{equation*} therefore~\eqref{4.5:km1} holds.

We want to show that if $\epsilon$ sufficiently small, then $\ff^\lambda(\tilde F) < \ff^\lambda(F)$, implying the claim with $G= \tilde F$.
By \cref{lemma3.1:np} \begin{equation}\label{4.6:km1} \begin{split} \ff^\lambda(\tilde F) = \ff^\lambda(l F_1) & \le l^{2n + q} \ff^\lambda(F_1) = \ff^\lambda(F_1) + (l^{2n + q} - 1)\ff^\lambda(F_1). \end{split} \end{equation}
Choosing $\epsilon \le 1$, we have $1 \le l \le 2^{\frac{1}{n}}$, and by Taylor's formula we obtain $l^{2n + q} - 1 = (1 + \gamma)^{2 + q/n} - 1 \le \gamma K$ for some $K > 0$ independent of $\gamma$, for $\epsilon$ sufficiently small.
By~\eqref{4.6:km1} we arrive at \begin{equation*} \ff^\lambda(\tilde F) - \ff^\lambda(F_1) \le \gamma K \ff^\lambda(F_1). \end{equation*}
By the definition of $\Sigma$ and since $\rr(F_1) + \rr(F_2) \le \rr(F)$ 
\begin{equation}\label{4.7:km1} \begin{split} \ff^\lambda(\tilde F) - \ff^\lambda(F) & \le \rr(F_1) + \g(F_1) + \rr(F_2) + \g(F_2) - \rr(F) - \g(F) + \Sigma - \ff^\lambda(F_2) + \gamma K \ff^\lambda(F_1) \\ & \le - \frac{1}{2} \ff^\lambda(F_2) + \gamma K \ff^\lambda(F_1). \end{split} \end{equation}
By positivity of $\rr$ and $\g$ and the isoperimetric inequality, we have $\ff^\lambda(F_2) > P_\lambda(F_2) \ge c(n, \lambda) m_2^{\frac{n - 1}{n}}$.
As in~\eqref{zz:eqStimaEnergia} we obtain $\ff^\lambda(F_1) \le \ff^\lambda(F)$.
By~\eqref{4.5:km1}, since $\gamma m \le 2 m_2$ and $\gamma \le \epsilon$,~\eqref{4.7:km1} turns into \begin{equation*}\label{4.8:km1} \begin{split} \ff^\lambda(\tilde F) - \ff^\lambda(F) & \le - c(n, \lambda) m_2^{\frac{n - 1}{n}} + \gamma K \ff^\lambda(F) \le - c(n, \lambda) m_2^{\frac{n - 1}{n}} + C(n, \lambda, g, G, q) \max\left\{m_2, \epsilon^{\frac{1}{n}} m_2^{\frac{n - 1}{n}}\right\}. \end{split} \end{equation*}
Since $m_2 \le \epsilon$ by~\eqref{4.4:km1}, for $\epsilon$ sufficiently small the assertion of the lemma holds with $G = \tilde F$. \end{proof}

Next lemma is an improvement of the standard density estimate for quasiminimizers. 
\begin{lemma}\label{lemma4.3:km}
Let $g$ be $\rr$-admissible, $q$-growing and infinitesimal, and let $G$ be $\g$-admissible. Then there exists $c = c(n, \lambda, g, G, q)>0$ such that the following holds. 
Let $E \subset \R^n \setminus H$ be a minimizer of $\ff^\lambda$ with $|E| = m$, $m > 0$.
Then for almost every $x \in E$ there holds \begin{equation*}\label{4.9:km} |E \cap B_{1}(x)| \ge c \min \{1, m\}. \end{equation*} \end{lemma}

\begin{proof}
For $r > 0$ and $x \in E$, let $F_1^r := E \setminus B_r(x)$ and $F_2^r := E \cap B_r(x)$.
Note that $|F_1^r| + |F_2^r| = m$ and $|F_2^r| \le \omega_n r^n$.
Then there exists $C > 0$, depending on $n$, $\lambda$, $g$, $G$ and $q$, such that~\eqref{4.4:km1} holds for all $r \le r_1 := C \min \left\{1, \sqrt[n]{m}\right\}$.
Note that we can choose $C \le 1$.
Since $E$ is a minimizer, \cref{lemma4.2:km} implies that~\eqref{4.3:km1} cannot be satisfied for any $r \le r_1$.
Equivalently, recalling also \cite[Corollary 2.5]{PascalePozzettaQuantitative}, for all $r \le r_1$ we have \begin{equation}\label{4.10:km1} \Sigma^r := P_\lambda(F_1^r) + P_\lambda(F_2^r) - P_\lambda(E) > \frac{1}{2} \ff^\lambda (F_2^r) > \frac{1}{2} P_\lambda(F_2^r) \ge \frac{1 - \lambda}{4} P(F_2^r). \end{equation}
At the same time,
for almost every $r$ we have \begin{equation*} \Sigma^r = 2 \hh^{n - 1}(E^{(1)} \cap \partial B_r(x)).
\end{equation*}
By~\eqref{4.10:km1}, for a constant $c(n, \lambda) \in \left(0, \frac{1}{2}\right)$ there holds
\begin{equation}\label{4.11:km1} 2\hh^{n - 1}(E^{(1)} \cap \partial B_r(x)) > c(n, \lambda) \, (\hh^{n - 1}(\partial^*E \cap B_r(x)) + \hh^{n - 1}(E^{(1)}\cap \partial B_r(x))),\end{equation}
for almost every $r$.
Let us now distinguish two cases.
If there exists $r_2 \in \left(\frac{r_1}{2}, r_1\right)$ such that $|E \cap B_{r_2}| \ge \frac12 \omega_n r_2^n$, then by the choice of $r_1$ we get \begin{equation*} |E \cap B_1| \ge |E \cap B_{r_2}| \ge c(n) \left(\frac{r_1}{2}\right)^n = c(n, \lambda, g, G, q) \min\{1, m\}, \end{equation*} and the proof is concluded. \\
Let us assume that $|E \cap B_r| < \frac12 \omega_n r^n$ for all $r \in \left(\frac{r_1}{2}, r_1\right)$.
Then we rearrange terms in~\eqref{4.11:km1} and apply the relative isoperimetric inequality \cite[Proposition 12.37]{MaggiBook} to the right-hand side to obtain \begin{equation*}\label{4.12:km1} \hh^{n - 1}(E^{(1)} \cap \partial B_r(x)) \ge c(n, \lambda) |E \cap B_r(x)|^{\frac{n - 1}{n}}. \end{equation*}
Let us denote $U(r) := |E \cap B_r(x)|$.
Then $\frac{\d U(r)}{\d r} = \hh^{n - 1}(E^{(1)} \cap \partial B_r(x))$ for all $r \in \left(\frac{r_1}{2}, r_1\right)$ and \begin{equation*}\label{4.13:km1} \frac{\d U(r)}{\d r} \ge c(n, \lambda) U^{\frac{n - 1}{n}}(r) \qquad \forall  r \in \left(\frac{r_1}{2}, r_1\right). \end{equation*}
For $\hh^n$-a.e. $x \in E$, we have $U(r) > 0$ for all $r > 0$ and ODE comparison in $r \in \left(\frac{r_1}{2}, r_1\right)$ implies that \begin{equation*} U^{1/n}(r) \ge U^{1/n}\left(\frac{r_1}{2}\right) + c(n, \lambda) \left(r-\frac{r_1}{2}\right) \ge c(n, \lambda) \left(r-\frac{r_1}{2}\right) \qquad \forall r \in \left(\frac{r_1}{2}, r_1\right). \end{equation*} 
Then the lemma follows as
\begin{equation*} c(n, \lambda, g, G) \min\{1, m\} = c(n, \lambda) r_1^n \le U(r_1) = |E \cap B_{r_1}(x)| = \left|E \cap B_{C \min \left\{1, \sqrt[n]{m}\right\}}(x)\right| \le |E \cap B_{1}(x)|. \end{equation*} \end{proof}

\begin{remark} We remark that the density estimate in \cref{lemma4.3:km} is more precise than the one provided in \cref{subsec:3.2}.
Indeed, in \cref{lemma:quasiminimal} $K$ and $r_0$ depend on the minimizer, and consequently $c$ in \cref{density:quasiminimal} also inherits this dependence.
At the same time, \cref{lemma4.3:km} requires $g$ to be $q$-growing, which is not required in \cref{lemma:quasiminimal}. \end{remark}

The following lemma will imply \cref{theorem:2} for $\beta \in(0, 1)$.

\begin{lemma}\label{lemma7.2:km} Let \begin{equation*}
g(x) = \frac{1}{|x|^\beta}, \qquad 0 < \beta < n, \, x \in \R^n \setminus \{0\} ,
\end{equation*}
and let $G$ be $\g$-admissible.
Let $E$ be a minimizer of $\ff^\lambda$ with $|E| = m$ and $m \ge 1$.
Then
\begin{equation}\label{7.3:km} 
cm^{\frac{1}{\beta}} \le {\rm diam} \, E \le C m,
\end{equation}
 for some $C$, $c > 0$ depending only on $n$, $\beta$, $\lambda$, $G$. 
\end{lemma}

\begin{proof} 
By \cref{teo:limitato} and \cref{teo:connesso} we know that $E$ is essentially bounded and indecomposable.
In particular $d := {\rm diam}\, E < \infty$.
By \cref{theorem3.4:km} we get the existence of $c(n, \lambda, \beta, G) > 0$ such that \begin{equation*}\label{7.4:km} \begin{split} \frac{m^2}{d^\beta} & \le \int_E\int_E \frac{1}{|x - y|^\beta}\de y \de x = \rr(E) \le \ff^\lambda(E) \le cm, \end{split} \end{equation*} which implies the first bound in~\eqref{7.3:km}.
\newline In order to prove the upper bound in~\eqref{7.3:km}, we may clearly assume that $\frac{d}{\sqrt{2}} > 3$. Recalling that we identify $E$ with the bounded open set $E^{(1)}$, we let $x^{(1)}$, $x^{(2)} \in \bar E$ such that \begin{equation*} |x^{(1)} - x^{(2)}| = d. \end{equation*}
Up to a rotation with respect to an axis orthogonal to $\{x_n = 0\}$, we can write \begin{equation*} x^{(2)} - x^{(1)} = \Braket{x^{(2)} - x^{(1)}, e_1} e_1 + \Braket{x^{(2)} - x^{(1)}, e_n} e_n. \end{equation*}
In particular, \begin{equation*} \max \left\{\left|\Braket{x^{(2)} - x^{(1)}, e_1}\right|, \left|\Braket{x^{(2)} - x^{(1)}, e_n}\right|\right\} \ge \frac{d}{\sqrt{2}}. \end{equation*}
Assume for simplicity that \begin{equation*} \left|\Braket{x^{(2)} - x^{(1)}, e_n}\right| \ge \frac{d}{\sqrt{2}}, \end{equation*} the remaining case being analogous.
Up to relabeling, assume also that \begin{equation*} \Braket{x^{(2)}, e_n} > \Braket{x^{(1)}, e_n} \end{equation*}
Let $N$ be the largest integer smaller than $\frac{d}{3\sqrt{2}}$, i.e. $N := \left\lfloor\frac{d}{3\sqrt{2}}\right\rfloor$.
Since $E$ is indecomposable, for every $j = 1, \dots, N$ there holds \begin{equation*} \left|E \cap \left\{3 j -1 + \Braket{x^{(1)}, e_n} < x_n < 3 j + \Braket{x^{(1)}, e_n}\right\}\right| > 0. \end{equation*} 
For every $j = 1, \dots, N$, let \begin{equation*} x_j \in E \cap \left\{3 j -1 + \Braket{x^{(1)}, e_n} < x_n < 3 j + \Braket{x^{(1)}, e_n}\right\}. \end{equation*} 
The balls $B_{1}(x_j)$, $j = 1, \dots, N$, are pairwise disjoint and, for a suitable choice of $x_j$, we can apply \cref{lemma4.3:km} to get \begin{equation*}\label{7.5:km} m = |E| \ge \sum_{j = 1}^N |B_{1}(x_j) \cap E| \ge c(n, \lambda, \beta, G) \, N \ge c(n, \lambda, \beta, G) \, d. \end{equation*} \end{proof}

The following lemma will imply \cref{theorem:2} for $\beta = 1$.

\begin{lemma}\label{lemma7:fn} Let \begin{equation*} g(x) = \frac{1}{|x|^\beta}, \qquad \beta \in (0, n), \, x \in \R^n \setminus \{0\} \end{equation*} and let $G$ be $\g$-admissible.
Let $E \subset \R^n \setminus H$ be a minimizer for $\ff^\lambda$ with $|E| = m$, $m > 0$.
Then \begin{equation*} \int_E \int_E \frac{1}{|x - y|^{\beta - 1}} \de y \de x \le c(n)\, m. \end{equation*} \end{lemma}

\begin{proof}
Let $\nu \in \mathbb{S}^{n - 1}\setminus\{\pm e_n\}$ and $t \in \R$.
Denote \begin{equation*} \begin{split} & E^+_{\nu, t} := E \cap \{\Braket{\nu, x} > t\} \\ & E^-_{\nu, t} := E \cap \{\Braket{\nu, x} < t\}. \end{split} \end{equation*}
Let $\nu_1 := \frac{\nu_{hor}}{|\nu_{hor}|}$, where $\nu_{hor}$ is the orthogonal projection of $\nu$ on $\{x_n = 0\}$.
For any $\rho \ge 0$, the set \begin{equation*} E^+_{\nu, t} \cup (E^-_{\nu, t} - \rho \nu_1) \end{equation*} has measure $m$ and, by minimality of $E$, \begin{equation}\label{13:fn} \ff^\lambda(E^+_{\nu, t} \cup (E^-_{\nu, t} - \rho \nu_1)) \ge \ff^\lambda(E). \end{equation}
For any $\rho > 0$ and for a.e. $t \in \R$ \begin{equation*} P_\lambda(E^+_{\nu, t} \cup (E^-_{\nu, t} - \rho \nu_1)) = P_\lambda (E^+_{\nu, t}) + P_\lambda (E^-_{\nu, t}) \le P_\lambda(E) + 2 \hh^{n - 1}(E \cap \{\Braket{\nu, x} = t\}). \end{equation*}
For any $\rho \ge 0$ we have
\begin{equation*} \begin{split} \int_{E^+_{\nu, t} \cup (E^-_{\nu, t} - \rho \nu_1)} \int_{E^+_{\nu, t} \cup (E^-_{\nu, t} - \rho \nu_1)} \frac{1}{|x - y|^\beta}\de y \de x = & \int_{E^+_{\nu, t}}\int_{E^+_{\nu, t}} \frac{1}{|x - y|^\beta}\de y \de x + \int_{E^-_{\nu, t}}\int_{E^-_{\nu, t}} \frac{1}{|x - y|^\beta}\de y \de x \\ & + 2 \int_{E^+_{\nu, t}}\int_{E^-_{\nu, t}}\frac{1}{|x - y + \rho \nu_1|^\beta}\de y \de x. \end{split} \end{equation*}
Moreover
\begin{equation*} \int_{E^+_{\nu, t}}\int_{E^-_{\nu, t}}\frac{1}{|x - y + \rho \nu_1|^\beta}\de y \de x \to 0 \end{equation*} as $\rho \to \infty$.
Hence, by~\eqref{13:fn}, letting $\rho \to \infty$, we get
\begin{equation*} \begin{split} & P_\lambda(E) + 2 \hh^{n - 1}(E \cap \{\Braket{\nu, x} = t\}) + \int_{E^+_{\nu, t}}\int_{E^+_{\nu, t}} \frac{1}{|x - y|^\beta}\de y \de x + \int_{E^-_{\nu, t}}\int_{E^-_{\nu, t}} \frac{1}{|x - y|^\beta}\de y \de x + \g(E^+_{\nu, t}) + \g(E^-_{\nu, t}) \\ & \ge P_\lambda(E) + \int_{E^+_{\nu, t}}\int_{E^+_{\nu, t}} \frac{1}{|x - y|^\beta}\de y \de x + \int_{E^-_{\nu, t}}\int_{E^-_{\nu, t}} \frac{1}{|x - y|^\beta}\de y \de x + 2 \int_{E^+_{\nu, t}}\int_{E^-_{\nu, t}} \frac{1}{|x - y|^\beta}\de y \de x + \g(E). \end{split} \end{equation*}
Then
\begin{equation*} \begin{split} \hh^{n - 1}(E \cap \{\Braket{\nu, x}\}= t) & \ge \int_{E^+_{\nu, t}}\int_{E^-_{\nu, t}} \frac{1}{|x - y|^\beta}\de y \de x \\ &
= \int_E \int_E
\chi_{\{\Braket{\nu,\cdot}<t\}}(y) \chi_{\{\Braket{\nu,\cdot}>t\}}(x)
\frac{1}{|x - y|^\beta}\de y \de x. \end{split} \end{equation*}
Integrating the last inequality with respect to $t \in \R$, by Fubini's theorem we get \begin{equation*} 
\begin{split}
m &\ge
\int_E \int_E \int_{-\infty}^{+\infty} \chi_{\{\Braket{\nu,\cdot}<t\}}(y) \chi_{\{\Braket{\nu,\cdot}>t\}}(x) \de t \frac{1}{|x - y|^\beta}\de y \de x
=
\int_E \int_E \int_{-\infty}^{+\infty} 
\chi_{(\Braket{\nu,y}, \Braket{\nu, x})}(t)
\de t \frac{1}{|x - y|^\beta}\de y \de x
\\
&=\int_E \int_E \frac{\Braket{\nu, x - y}_+}{|x - y|^\beta}\de y \de x.    
\end{split}
\end{equation*}
Further integrating over $\mathbb{S}^{n - 1} \setminus \{\pm e_n\}$, since
\begin{equation*} \int_{\mathbb{S}^{n - 1}} \Braket{\nu, x - y}_+ \de \nu = c(n) |x - y|, \end{equation*} 
by symmetry of $\mathbb{S}^{n-1}$, we conclude that
\[
|\mathbb{S}^{n-1}| m \ge
\int_E \int_E \int_{\mathbb{S}^{n-1}} \Braket{\nu, x - y}_+ \de \nu \frac{1}{|x - y|^\beta}\de y \de x
= c(n) \int_E \int_E  \frac{1}{|x - y|^{\beta-1}}\de y \de x.
\]
\end{proof}

The following lemma will imply \cref{theorem:2} for $\beta \in (1,2]$.

\begin{lemma}\label{lemma8:fn} Let \begin{equation*} g(x) = \frac{1}{|x|^\beta}, \qquad 0 < \beta < n, \, x \in \R^n \setminus \{0\} \end{equation*} and let $G$ be $\g$-admissible.
Let $E \subset \R^n \setminus H$ be a minimizer for $\ff^\lambda$ with $|E| = m$, $m > \omega_n$.
Then, for $1 \le r \le \frac{{\rm diam} E}{2}$, \begin{equation*} |E \cap B_r(x)| \ge c(n, \lambda, \beta, G) \, r \qquad {\rm for \, a.e.}\, x \in E. \end{equation*} \end{lemma}

\begin{proof} Let $N := \left\lfloor\frac{r - 1}{3}\right\rfloor$ and $x \in E$.
If $r < 4$, by \cref{lemma4.3:km} \begin{equation*} |E \cap B_r(x)| \ge |E \cap B_{1}(x)| \ge c(n, \lambda, \beta, G) = \frac{c(n, \lambda, \beta, G)}{4}\,4 \ge c(n, \lambda, \beta, G) \, r. \end{equation*}
So we can assume that $r \ge 4$, in particular $N \ge 1$.
Since $E$ is indecomposable by \cref{teo:connesso}, for every $i = 0, \dots, N - 1$ there holds \begin{equation*} |E \cap (B_{3i + 3}(x) \setminus B_{3i + 2}(x))| > 0. \end{equation*} 
For every $i = 0, \dots, N - 1$, let \begin{equation*} y_i \in E \cap (B_{3i + 3}(x) \setminus B_{3i + 2}(x)). \end{equation*}
The balls $B_1(y_i)$, $i = 0, \dots N - 1$, are pairwise disjoint and, for a suitable choice of $y_i$, by \cref{lemma4.3:km} there exists $c(n, \lambda, \beta, G)$ such that $|E \cap B_{1}(y_i)| \ge c$ for $i = 0, \dots, N - 1$.
Finally
\begin{equation*} |E \cap B_r(x)| \ge \sum_{i = 0}^{N - 1} |E \cap B_{1}(y_i)| + |E \cap B_{1}(x)| \ge (N + 1) c \ge c \frac{r - 1}{3} \ge c \,r. \end{equation*} \end{proof}

Now we are ready to prove \cref{theorem:2}.
\begin{proof}[Proof of \cref{theorem:2}] \cref{lemma7.2:km} and \cref{lemma7:fn} easily imply \cref{theorem:2} for $\beta \in (0, 1]$ and mass $m$ sufficiently large.
Then it remains to consider $\beta \in (1, 2]$.
Let $E \subset \R^n \setminus H$ be a volume constrained minimizer for $\ff^\lambda$ with $|E|=m$. By \cref{lemma7.2:km}, for $m$ large enough, we can assume that ${\rm diam} \, E > 4$.
We observe first that
\begin{equation}\label{eq:zzStimaIntornoCircaTubolare}
     \frac{1}{r^{\beta -1}}\left|\{(x, y) \in E \times E \st |x - y| < r\}\right| =
     \frac{1}{r^{\beta-1}} \int_E |E \cap B_r(x)| \de x \le \frac{\omega_n r^n}{r^{\beta-1}} |E| \xrightarrow[r\to0]{}0.
     %
\end{equation}
Applying the coarea formula on $\R^n\times \R^n$ for the Lipschitz function $f:\R^n\times \R^n\to \R$ given by $f(x,y):=|x-y|$, observing that $|\nabla f | = \sqrt{2}$ and that
\[
\frac{\d}{\d r} \left|\{(x, y) \in E \times E \st |x - y| < r\}\right| = \frac{1}{\sqrt{2}} \int_{\{(x,y) \in \R^n\times \R^n\st |x-y|=r\}} \chi_{E\times E}(x,y)\de \hh^{2n-1}(x,y) \quad\text{for a.e. $r>0$,} 
\]
and integrating by parts, we estimate
\begin{equation*} \begin{split} 
\int_E \int_E \frac{1}{|x - y|^{\beta - 1}}\de y \de x &
= \frac{1}{\sqrt{2}}
\int_{\R^n\times \R^n} \frac{\chi_{E\times E}(x,y)}{|x - y|^{\beta - 1}} |\nabla f| \de x \de y
\\
&= \frac{1}{\sqrt{2}} \int_0^{+\infty} \int_{\{(x,y) \in \R^n\times \R^n\st |x-y|=r\}} \frac{\chi_{E\times E}(x,y)}{r^{\beta-1}} \de \hh^{2n-1}(x,y) \de r
\\
&= \lim_{\epsilon\to0^+} \int_\epsilon^{+\infty} \frac{1}{r^{\beta-1}} \frac{\d}{\d r} \left|\{(x, y) \in E \times E \st |x - y| < r\}\right|  \de r
\\
&= \lim_{\epsilon\to0^+} -\frac{1}{\epsilon^{\beta -1}}\left|\{(x, y) \in E \times E \st |x - y| < \epsilon\}\right| +\\
&\qquad\qquad- \int_{\epsilon}^{+\infty} (1-\beta) \frac{1}{r^\beta} |\{(x, y) \in E \times E \st |x - y| < r\}|  \de r.
\end{split}
\end{equation*}
Exploiting \eqref{eq:zzStimaIntornoCircaTubolare} and \cref{lemma8:fn} we deduce
\begin{equation*}
    \begin{split}
      \int_E \int_E \frac{1}{|x - y|^{\beta - 1}}\de y \de x 
      &= (\beta - 1) \int_0^{+\infty} |\{(x, y) \in E \times E \st |x - y| < r\}| \frac{\de r}{r^\beta} \\ & = (\beta - 1) \int_0^{+\infty} \int_E \frac{|E \cap B_r(x)|}{r^\beta} \de x \de r \\ & \ge (\beta - 1) \int_{1}^{\frac{{\rm diam} \, E}{2}} \int_E \frac{|E \cap B_r(x)|}{r^\beta} \de x \de r \\ & \ge c \int_{1}^{\frac{{\rm diam} \, E}{2}} \frac{|E|}{r^{\beta - 1}} \de r. \end{split} \end{equation*}
The final right-hand side in the previous chain of inequalities is bounded from below by $c\; |E|\; ({\rm diam} \, E)^{2 - \beta}$ if $\beta < 2$, and by $c\;|E|\;\log({\rm diam} \, E)$ if $\beta = 2$.
Combining this bounds with \cref{lemma7:fn}, we get a contradiction for $|E|$ large enough.
\end{proof}

\subsection{Absence of holes in minimizers}

As a corollary of the estimates proved in the last section, we prove here a further qualitative property of volume constrained minimizers of $\ff^\lambda$. The next theorem essentially tells that volume constrained minimizers of $\ff^\lambda$ do not have ``interior holes''.

\begin{theorem}\label{teo:nobuchi}
Let $g$ be $\rr$-admissible, $0$-growing, infinitesimal and symmetric and let $G$ be $\g$-admissible and coercive.
There exists $\bar m > 0$, depending on $n$, $\lambda$, $g$ and $G$, such that, for all $m \in (0, \bar m)$, every minimizer $E$ of $\ff^\lambda$ with $|E| = m$ has the following property. There is no set $F \subset \R^n \setminus (H \cup E)$ with $|F| > 0$ such that \begin{equation}\label{np:4.4} P_\lambda(E) = P_\lambda(E \cup F) + P(F, \R^n \setminus H) + \lambda \hh^{n - 1}(\partial^*F \cap \partial H). \end{equation} \end{theorem}

We begin with a preparatory lemma.

\begin{lemma}\label{vertical:boundedness2}
Let $g$ be $\rr$-admissible, $q$-growing, infinitesimal and symmetric and let $G$ be $\g$-admissible and coercive.
There exist $\bar m > 0$ and $\bar T > 0$, depending on $n$, $\lambda$, $g$, $G$, $q$ such that, for all $m \in (0, \bar m)$, every volume constrained minimizer $E$ of $\ff^\lambda$ with $|E| = m$ satisfies
\begin{equation*} |E \cap \{x_n > \bar T\}| = 0. \end{equation*}
\end{lemma}

\begin{proof} By \cref{vertical:boundedness} there exists $\bar T_E < \infty$, depending on $n$, $\lambda$, $g$, $G$ and $E$ with \begin{equation*} \bar T_E := \sup \{ t \st |E \cap \{x_n>t\}| >0 \}. \end{equation*}
Let $x_E \in E$ such that \begin{equation*} (x_E)_n \ge \frac{1}{2} \bar T_E. \end{equation*}
By \cref{lemma4.3:km} there exists $c = c(n, \lambda, g, G, q) > 0$ such that, if $\bar m < 1$, then
\begin{equation*} |E \cap B_1(x_E)| \ge c \,m. \end{equation*}
Therefore
\begin{equation}\label{limitatezza:1} \ff^\lambda(E) \ge P_\lambda(E) + \int_{E \cap B_1(x_E)} G \ge
P_\lambda(B^\lambda(m))
+ c m \inf_{((x_E)_n-1, (x_E)_n+1)} G. \end{equation}
On the other hand, by \cref{lemma2.3:np}, if $\bar m \le |B^\lambda|$ we have \begin{equation}\label{limitatezza:2} \ff^\lambda(E) \le \ff^\lambda(B^\lambda(m)) \le P_\lambda(B^\lambda(m)) 
+ c(n, \lambda, g, G) m. \end{equation}
Putting together~\eqref{limitatezza:1} and~\eqref{limitatezza:2} we obtain \begin{equation*} \inf_{((x_E)_n-1, (x_E)_n+1)} G \le c(n, \lambda, g, G, q). \end{equation*}
Since $G$ is coercive, then $(x_E)_n$, and in particular also $\bar T_E$, is bounded by a constant independent of $E$, and we conclude the proof.
\end{proof}
\begin{remark} We remark that \cref{vertical:boundedness2} is a stronger result than \cref{vertical:boundedness}.
Indeed, the bound in \cref{vertical:boundedness2} does not depend on the minimizer.
At the same time, \cref{vertical:boundedness} does not require that $g$ is $q$-growing and $G$ is coercive. \end{remark}

Now we are ready to prove \cref{teo:nobuchi}.

\begin{proof}[Proof of \cref{teo:nobuchi}] If $\bar m$ is sufficiently small, \cref{vertical:boundedness2} guarantees that there exists $\bar T > 0$, depending on $n$, $\lambda$, $g$, $G$, such that \begin{equation*} |E \cap \{x_n > \bar T\}| = 0. \end{equation*} 
Assume that there exists a set $F \subset \R^n \setminus (H \cup E)$ with $v := |F| > 0$ and such that~\eqref{np:4.4} holds.
We aim to find a contradiction if $\bar m$ is chosen suitably small.
Let $\bar m \le |B^\lambda|$.
By the minimality of $E$, the isoperimetric inequality~\eqref{isoperimetric:inequality}, the relative isoperimetric inequality outside convex sets \cite{ChoeGhomiRitoreInequality, FuscoMorini} and since $P_\lambda(B^\lambda(m)) = n |B^\lambda|^{\frac{1}{n}} m^{\frac{n - 1}{n}}$ \cite[Lemma 3.1]{PascalePozzettaQuantitative}, we find \begin{equation*} \begin{split} n|B^\lambda|^{\frac{1}{n}} m^{\frac{n - 1}{n}} + \rr(B^\lambda (m)) + \g(B^\lambda (m)) & = \ff^\lambda(B^\lambda (m)) \ge \ff^\lambda(E) \ge P_\lambda(E) \\ & = P_\lambda(E \cup F) + P(F, \R^n \setminus H) + \lambda \hh^{n - 1}(\partial^*F \cap \partial H) \\ & \ge P_\lambda(E \cup F) + (1 - |\lambda|) P(F, \R^n \setminus H) \\ & \ge n |B^\lambda|^{\frac{1}{n}} (m + v)^{\frac{n - 1}{n}} + (1 - |\lambda|)\, n \left(\frac{\omega_n}{2}\right)^{\frac{1}{n}}v^{\frac{n - 1}{n}} \\ & \ge n |B^\lambda|^{\frac{1}{n}} m^{\frac{n - 1}{n}} + (1 - |\lambda|)\, n \left(\frac{\omega_n}{2}\right)^{\frac{1}{n}}v^{\frac{n - 1}{n}} \end{split} \end{equation*}
which gives, by \cref{lemma2.3:np}, \begin{equation}\label{np:4.5} v \le \left(\frac{\rr(B^\lambda (m)) + \g(B^\lambda (m))}{(1 - |\lambda|) \, n \sqrt[n]{\frac{\omega_n}{2}}}\right)^{\frac{n}{n - 1}} \le c(n, \lambda, g, G) \left(\frac{m}{(1 - |\lambda|) \, n \sqrt[n]{\frac{\omega_n}{2}}}\right)^{\frac{n}{n - 1}}. \end{equation}
Since $\bar m < c(n, \lambda, g, G)$, also $v$ is bounded by a suitable $\bar v(n, \lambda, g, G)$. 
By \cite[Lemma 3.5]{NovagaPratelli} there exists a continuous and increasing function $\phi : (0, \infty) \to (0, \infty)$, with $\phi(0) = 0$, such that for every two sets $F_1$, $F_2 \subset \R^n \setminus H$ one has \begin{equation*} \rr(F_1, F_2) \le |F_1| \, \phi(|F_2|). \end{equation*}
Then \begin{equation*} \begin{split} \rr(E \cup F) - \rr(E) & = \rr(F, F) + 2\rr(F, E) \le v\, \phi(\bar v) + 2 v\, \phi(\bar m) \le c(n, \lambda, g, G) v. \end{split} \end{equation*}
Let us prove that also $F$ is essentially contained in $\{0 < x_n \le \bar T\}$. 
To this end, assume by contradiction that \begin{equation}\label{contradiction:misura} \hh^{n - 1}\left(\partial^*F \setminus \left(H \cup \partial^*E\right)\right) > 0. \end{equation}
Let us denote \begin{equation*} \begin{split} & \Sigma_E := \hh^{n - 1}\left(\partial^*E \setminus \left(H \cup \partial^*F\right)\right) \\ & \Sigma := \hh^{n - 1}\left(\left(\partial^*E \cap \partial^*F\right) \setminus H\right) \\ & \Sigma_F := \hh^{n - 1}\left(\partial^*F \setminus \left(H \cup \partial^*E\right)\right) \\ & \Theta_E :=  \hh^{n - 1}\left(\partial^*E \cap \partial H\right) \\ & \Theta_F :=  \hh^{n - 1}\left(\partial^*F \cap \partial H\right). \end{split} \end{equation*}
By~\eqref{np:4.4} we obtain \begin{equation*} \begin{split} \Sigma_E + \Sigma - \lambda \Theta_E & = P_\lambda(E) \\ & = P_\lambda(E \cup F) + P(F, \R^n \setminus H) + \lambda \Theta_F \\ & = \Sigma_E + \Sigma_F - \lambda \Theta_E - \lambda \Theta_F + \Sigma + \Sigma_F + \lambda\Theta_F \\ & = \Sigma_E + 2\Sigma_F  + \Sigma - \lambda \Theta_E. \end{split} \end{equation*}
In particular, we get $\Sigma_F = 0$, contradicting~\eqref{contradiction:misura}.
Therefore $F \subset \{0 < x_n \le \bar T\}$ and we also deduce
\begin{equation}\label{eq:zzBoundG}
\g(E \cup F) - \g(E) = \g(F) \le \int_F \sup_{(0, \bar T)} G\de x = c(n, \lambda, g, G)\, v. \end{equation}
By~\eqref{np:4.4}, we obtain \begin{equation}\label{4.6:np} \begin{split} \ff^\lambda(E \cup F) & = P_\lambda(E \cup F) + \rr(E \cup F) + \g(E \cup F) \\ & = P_\lambda(E) - P(F, \R^n \setminus H) - \lambda \hh^{n - 1}(\partial^*F \cap \partial H) + \rr(E \cup F) + \g(E \cup F) \\ & \le P_\lambda(E) - (1 - |\lambda|) P(F, \R^n \setminus H) + \rr(E \cup F) + \g(E \cup F) \\ & \le P_\lambda(E) - (1 - |\lambda|)\, n \left(\frac{\omega_n}{2}\right)^{\frac{1}{n}} v^{\frac{n - 1}{n}} + \rr(E) + c \,v + \g(E) + c \, v \\ & = \ff^\lambda(E) - (1 - |\lambda|)\, n \left(\frac{\omega_n}{2}\right)^{\frac{1}{n}} v^{\frac{n - 1}{n}} + c\,v < \ff^\lambda(E), \end{split} \end{equation} where the last inequality holds if $v$ is sufficiently small, hence by~\eqref{np:4.5} as soon as $\bar m$ is sufficiently small. \newline
Let $t \in (0, \infty)$ be such that, if $D = \left\{x \in E \cup F, x_n < t\right\}$, then $|D| = m$.
Clearly $\ff^\lambda(D) \le \ff^\lambda(E\cup F)$, hence \eqref{4.6:np} implies  $\ff^\lambda(D) < \ff^\lambda(E)$, contradicting the minimality of $E$.
%
\end{proof}

\begin{remark}
Note that if the function $G$ were globally bounded, \cref{teo:nobuchi} could be easily extended to minimizers of $\ff^\lambda$ in the class \begin{equation*} \mathcal{A}_{m} = \{\Omega \subset \R^n \setminus H \, {\rm measurable} \st |\Omega| = m\} \end{equation*}
{\em for every} mass $m > 0$.\\
Indeed, in the proof of \cref{teo:nobuchi} we exploited \cref{vertical:boundedness2} just to get the estimate \eqref{eq:zzBoundG}, which is trivial in case $G$ were assumed to be globally bounded.
\end{remark}

Now we are ready to complete the proof of \cref{theorem:1}.

\begin{proof}[Proof of \cref{theorem:1}] \cref{theorem:1} follows by \cref{simple:existence}, \cref{remark:quantitative}, \cref{prop:existence}, \cref{teo:limitato}, \cref{teo:connesso} and \cref{teo:nobuchi}.
\end{proof}

\section{Generalized minimizers}\label{sec:5}

Let us give the following definition.

\begin{definition}\label{def:EpsilonEnergy} If $E \subset \R^n \setminus H$ is a measurable set, $g$ is a $\rr$-admissible function, $G$ is a $\rr$-admissible function and $\epsilon_1$, $\epsilon_2 > 0$, we define the functional \begin{equation*} \begin{split} \ff^\lambda_\epsilon(E) &:= P_\lambda(E) + \epsilon_1\rr(E) + \epsilon_2 \g(E) \\ &= P_\lambda(E) + \epsilon_1 \int_E \int_E g(y - x) \de y \de x + \epsilon_2 \int_E G(x_n) \de x. \end{split} \end{equation*} \end{definition}

\begin{remark}\label{remark:scaling}
We remark that minimizing the functional $\ff^\lambda$ in the small mass regime is equivalent to minimizing the functional $\ff^\lambda_\epsilon$ for $\epsilon_1$, $\epsilon_2$ small and among sets of a fixed volume.
Indeed, let for instance $|E| = |B^\lambda|$, if $m > 0$ and $\bar \epsilon := \frac{m^{\frac{1}{n}}}{|B^\lambda|^{\frac{1}{n}}}$, then $\tilde E := \bar \epsilon E$ has volume $m$ and by scaling we have
\begin{equation*} \ff^\lambda(\tilde E) = \bar\epsilon^{n - 1} \left(P_\lambda(E) + \bar\epsilon^{n + 1} \int_E\int_E g(\bar\epsilon(y - x)) \de y \de x + \bar\epsilon \int_E G(\bar\epsilon x_n)\de x\right). \end{equation*}

In particular we deduce that 

\begin{itemize}
\item for every $\rr$-admissible function $g_1$, $\g$-admissible function $G_1$ and $m > 0$, there exist $\bar\epsilon > 0$, a $\rr$-admissible function $g_2$ and a $\g$-admissible function $G_2$ such that, if 
\begin{equation*} \ff^\lambda(E) = P_\lambda(E) + \int_E \int_E g_1(y - x) \de y \de x + \int_E G_1(x_n) \de x \end{equation*}
and
\begin{equation*} \ff^\lambda_{\bar\epsilon}(E) = P_\lambda(E) + \bar\epsilon^{n + 1} \int_E\int_E g_2(y - x) \de y \de x + \bar\epsilon \int_E G_2(x_n)\de x, \end{equation*}
then $\inf_{|E| = m} \ff^\lambda(E)$ is proportional to $\inf_{|E| = |B^\lambda|} \ff^\lambda_{\bar\epsilon}(E)$ and the variational problems are equivalent.
\item for every $\rr$-admissible function $g_2$, $\g$-admissible function $G_2$, $\epsilon_1, \epsilon_2 > 0$ and $m > 0$, there exist a $\rr$-ad\-mis\-si\-ble function $g_1$ and a $\g$-admissible function $G_1$ such that, if 
\begin{equation*} \ff^\lambda(E) = P_\lambda(E) + \int_E \int_E g_1(y - x) \de y \de x + \int_E G_1(x_n) \de x \end{equation*}
and
\begin{equation*} \ff^\lambda_\epsilon(E) = P_\lambda(E) + \epsilon_1 \int_E\int_E g_2(y - x) \de y \de x + \epsilon_2 \int_E G_2(x_n)\de x, \end{equation*}
then $\inf_{|E| = |B^\lambda|} \ff^\lambda_\epsilon(E)$ is proportional to $\inf_{|E| = m} \ff^\lambda(E)$ and the variational problems are equivalent.
\end{itemize}
\end{remark}
\vspace{2mm}

From now on for the rest of the section, we assume that $\epsilon_1, \epsilon_2>0$ and $g, G$ as in \cref{def:EpsilonEnergy} are given. 
Hence we also define the generalized energy corresponding to $\ff^\lambda_\epsilon$ as
\begin{equation*} \label{generalized:energy} \tilde\ff_{\epsilon}^\lambda(E) := \inf_{h \in \N}\tilde\ff^\lambda_{\epsilon, h}(E), \end{equation*} where \begin{equation*} \tilde\ff^\lambda_{\epsilon, h}(E) := \inf \left\{\sum_{i = 1}^h \ff^\lambda_{\epsilon}(E^i) \st E = \bigcup_{i = 1}^h E^i, E^i \cap E^j = \emptyset \quad {\rm for} \; 1 \le i \neq j \le h\right\}. \end{equation*}
\vspace{2mm}
The goal of this Section is to prove the following version of \cref{theorem:3}, suitably modified for the functional $\ff^\lambda_\epsilon$.
\begin{theorem}\label{generalized:existence} Let $g$ be $\rr$-admissible and $q$-growing and let $G$ be $\g$-admissible.
For every $\epsilon_1, \epsilon_2 > 0$ there exists a minimizer of $\tilde\ff^\lambda_\epsilon$ in the class \begin{equation*} \mathcal{A} := \left\{\Omega \subset \R^n\setminus H \; {\rm measurable} \st |\Omega| = |B^\lambda|\right\}. \end{equation*} 
More precisely, there exist a set $E \in \mathcal{A}$ and a subdivision $E = \cup_{j = 1}^h E^j$, with pairwise disjoint sets $E^j$, such that \begin{equation*} \tilde\ff^\lambda_\epsilon(E) = \sum _{j = 1}^h \ff^\lambda_\epsilon(E^j) = \inf \left\{\tilde\ff^\lambda_\epsilon(\Omega) \st \Omega \in \mathcal{A}\right\}. \end{equation*} 
Moreover, for every $1 \le j \le h$, the set $E^{j}$ is a minimizer of both the standard and the generalized energy for its volume, i.e. \begin{equation}\label{1.3:np} \tilde\ff^\lambda_\epsilon(E^{j}) = \ff^\lambda_\epsilon(E^{j}) = \min \left\{\tilde\ff^\lambda_\epsilon(\Omega) \st \Omega \subset \R^n\setminus H, \, |\Omega| = |E^{j}|\right\}. \end{equation} \end{theorem}

Note that an analogous version of \cref{lemma3.1:np} holds.
\begin{lemma}\label{lemma3.1:npbis} Let $E \subset \R^n \setminus H$ be a set of finite perimeter. 
Let $g$ be $\rr$-admissible and $q$-growing and let $G$ be $\g$-admissible. 
If $\alpha > 1$, then \begin{equation*} \ff^\lambda_\epsilon(\alpha E) \le \alpha^{2n + q} \ff^\lambda_\epsilon(E). \end{equation*} \end{lemma}

We begin by proving some preparatory lemmas.
The next geometric lemma allows to modify an excessively long and thin set decreasing its energy.
\begin{lemma}\label{lemma3.2:np} Let $g$ be $\rr$-admissible and $G$ be $\g$-admissible. 
For every $\bar m >0$ there exists $L(n, \lambda, \bar m) > 0$ such that the following holds. 
Let $E \subset \R^n \setminus H$, and let $a < b$ be two numbers with $b > a + 2 L$ and such that \begin{equation*} \left|\left\{x \in E \st a \le x_1 \le b\right\}\right| < \bar m. \end{equation*}
Then there exist two numbers $a^+ \in [a, a + L]$ and $b^- \in [b - L, b]$ such that, denoting $E^- = E \setminus ([a^+, b^-] \times \R^{n - 2}$ $\times$ $(0, \infty))$ and $m = |E| - |E^-| < \bar m$, one has \begin{equation}\label{3.1:np} \ff_\epsilon^\lambda(E^-) \le \ff_\epsilon^\lambda(E) - \frac{1}{2} n |B^\lambda|^{\frac{1}{n}} m^{\frac{n - 1}{n}}. \end{equation} \end{lemma}
\begin{proof} It is sufficient to prove the claim for bounded sets $E$ such that $\partial E \setminus \partial H$ is a smooth hypersurface with $\hh^{n - 1}(\{x \in \partial E \setminus H \st \nu^{E}(x) = \pm e_j\}) = 0$ for any $j = 1, \dots, n$.
Indeed, if $E$ is a generic set of finite perimeter satisfying the hypotheses of \cref{lemma3.2:np}, let $E_i \xrightarrow{L^1} E$ be the sequence of sets given by \cite[Lemma 2.4]{PascalePozzettaQuantitative}.
For $i$ sufficiently large, \begin{equation*} \left|\left\{x \in E_i \st a \le x_1 \le b\right\}\right| < \bar m \end{equation*} holds.
Then there exist $a_i^+ \in [a, a + L]$ and $b_i^-\in [b - L, b]$ such that, if we set $E_i^- = E_i \setminus \left([a^+_i, b^-_i] \times (0, \infty)\right)$ and $m_i = |E_i| - |E^-_i|$, we obtain \begin{equation*} \ff_\epsilon^\lambda(E^-_i) \le \ff_\epsilon^\lambda(E_i) - \frac{1}{2} n |B^\lambda|^{\frac{1}{n}} m_i^{\frac{n - 1}{n}}. \end{equation*}
Up to subsequence, $a_i^+$ and $b_i^-$ converge to certain $a^+ \in [a, a + L]$ and $b^-\in [b - L, b]$ respectively.
By \cite[Lemma 2.4]{PascalePozzettaQuantitative} $P_\lambda(E_i) \to P_\lambda(E)$.
By the lower semicontinuity of $P_\lambda$ (\cite[Lemma 3.7]{PascalePozzettaQuantitative}), the continuity of $\rr$ and $\g$ under strong $L^1$ convergence and the properties of $\{E_i\}$, if $E^- = E \setminus \left([a^+, b^-] \times (0, \infty)\right)$ then \begin{equation*} \begin{split} \ff_\epsilon^\lambda(E^-) & = P_\lambda(E^-) + \epsilon_1\rr(E^-) + \epsilon_2\g(E^-) \\ & \le \liminf_i (P_\lambda(E_i^-) + \epsilon_1 \rr(E^-_i) + \epsilon_2 \g(E^-_i)) \\ & \le \liminf_i \left(P_\lambda(E_i) + \epsilon_1 \rr(E_i) + \epsilon_2 \g(E_i) - \frac{1}{2} n |B^\lambda|^{\frac{1}{n}}m_i^{\frac{n - 1}{n}}\right) \\ & = \ff_\epsilon^\lambda(E) - \frac{1}{2} n |B^\lambda|^{\frac{1}{n}}m^{\frac{n - 1}{n}}. \end{split} \end{equation*}
So let us fix $\bar m$ and consider $a$ and $b$ as in the claim, with $L$ to be determined later. 
For $t \in \R$ let \begin{equation*} \sigma(t) := \hh^{n - 1} \left(E \cap \left\{x \in \R^n \st x_1 = t\right\}\right). \end{equation*}
If $c := \frac{a + b}{2}$, let \begin{equation*} \phi(t) := \int_t^c \sigma(s) \de s. \end{equation*}
We note that there exists $a^+ \in [a, a + L] \subset [a, c)$ such that, if \begin{equation*} m_1 := \left|\left\{x \in E \st a^+ < x_1 < c\right\}\right|, \end{equation*} then \begin{equation}\label{3.2:np} \sigma(a^+) \le \frac{1}{8} n |B^\lambda|^{\frac{1}{n}}m_1^{\frac{n- 1}{n}}. \end{equation}
Indeed, assume by contradiction that for every $t \in (a, a + L)$ it holds \begin{equation*} -\phi'(t) = \sigma(t) > \frac{1}{8} n |B^\lambda|^{\frac{1}{n}}\phi(t)^{\frac{n- 1}{n}}. \end{equation*}
Then $\phi|_{(a, a + L)}$ is a positive decreasing function satisfying \begin{equation*} \begin{cases} & \phi(a) \le \bar m, \\ & |\phi'(t)| > \frac{1}{8} n|B^\lambda|^{\frac{1}{n}} \phi(t)^{\frac{n - 1}{n}}. \end{cases} \end{equation*}
By standard ODE comparison there exists a constant $d > 0$ depending only on $n$, $\lambda$ and $\bar m$ such that, if $a+d < a+L$, then $\varphi(t)\to0$ as $t \to (a+d)^-$. Hence $L$ could be chosen so big that $a+d < a+L < c$, and then $\phi(t) = 0$ for any $t \in (a + d, c)$.
It follows that there exists $a^+$ such that~\eqref{3.2:np} holds.
Similarly, up to choosing a larger $L$, we have the existence of $b^- \in [b - L, b] \subset (c, b]$ such that, if \begin{equation*} m_2 := \left|\left\{x \in E \st c < x_1 < b^-\right\}\right|, \end{equation*} then \begin{equation*} \sigma(b^-) \le \frac{1}{8} n |B^\lambda|^{\frac{1}{n}}m_2^{\frac{n - 1}{n}}. \end{equation*}
Let $E^- := E \setminus \left(\left[a^+, b^-\right]\times(0, \infty)\right)$ and $F := E \setminus E^-$.
Then \begin{equation*} |F| = m = m_1 + m_2, \end{equation*} and, by isoperimetric inequality~\eqref{isoperimetric:inequality}, there holds
\begin{equation*} P_\lambda(F) \ge n |B^\lambda|^{\frac{1}{n}} m^{\frac{n - 1}{n}}. \end{equation*}
Hence
\begin{equation*} \begin{split} P_\lambda(E^-) & \le P_\lambda(E) - P_\lambda(F) + 2 (\sigma(a^+) + \sigma(b^-)) \\ & \le P_\lambda(E) - P_\lambda(F) + \frac{1}{4}n |B^\lambda|^{\frac{1}{n}} \left(m_1^{\frac{n - 1}{n}} + m_2^{\frac{n - 1}{n}}\right) \\ & \le P_\lambda(E) - P_\lambda(F) + \frac{1}{2} n |B^\lambda|^{\frac{1}{n}}(m_1 + m_2)^{\frac{n - 1}{n}} \\ & \le P_\lambda(E) - \frac{1}{2} n |B^\lambda|^{\frac{1}{n}}(m_1 + m_2)^{\frac{n - 1}{n}}. \end{split} \end{equation*} 
Since $E^- \subset E$, then $\rr(E^-) \le \rr(E)$ and $\g(E^-) \le \g(E)$, and we deduce~\eqref{3.1:np}. \end{proof}

The following variant of \cref{lemma3.2:np} concerns the case of the vertical direction when we modify a part lying on the hyperplane $\{x_n = 0\}$. 
\begin{lemma}\label{lemma3.2bis:np} Let $g$ be $\rr$-admissible and $G$ be $\g$-admissible. 
For every $\bar m \in \R$ there exists $L(n, \lambda, \bar m) > 0$ such that the following holds. 
Let $E \subset \R^n \setminus H$, and let $b$ be a number with $b > L$ and such that \begin{equation*} \left|\left\{x \in E \st 0 < x_n \le b\right\}\right| < \bar m. \end{equation*}
There exists then $b^- \in [b - L, b]$ such that, denoting $E^- = E \setminus \left(\R^{n - 1} \times [0, b^-]\right)$ and $m = |E| - |E^-| \le \bar m$, one has \begin{equation}\label{3.1bis:np} \ff_\epsilon^\lambda(E^-) \le \ff_\epsilon^\lambda(E) - \frac{1}{2} n |B^\lambda|^{\frac{1}{n}} m^{\frac{n - 1}{n}}. \end{equation} \end{lemma}
\begin{proof} As in the proof of \cref{lemma3.2:np}, we may assume that $\partial E \setminus \partial H$ is a smooth hypersurface with $\hh^{n - 1}(\{x \in \partial E \setminus H \st \nu^{E_i}(x) = \pm e_j\}) = 0$ for any $j = 1, \dots n$.
Let us fix $\bar m$ and consider $b$ as in the claim, with $L$ to be determined later. 
For almost every $t \in \R$, let \begin{equation*} \sigma(t) := \hh^{n - 1} \left(E \cap \left\{x \in \R^n \st x_1 = t\right\}\right). \end{equation*}
Let \begin{equation*} \phi(t) := \int_0^t \sigma(s) \de s. \end{equation*}
As in the proof of \cref{lemma3.2:np}, we note that there exists $b^- \in [b - L, b] \subset (0, b]$ such that, if \begin{equation*} m := \left|\left\{x \in E \st 0 < x_n < b^-\right\}\right|, \end{equation*} then \begin{equation*}\label{3.2bis:np} \sigma(b^-) \le \frac{1}{4} n |B^\lambda|^{\frac{1}{n}}m^{\frac{n- 1}{n}}. \end{equation*}
If $E^- := E \setminus \left(\R^{n - 1} \times \left[0, b^-\right]\right)$ and $F := E \setminus E^-$,
then \begin{equation*} |F| = m \end{equation*} and, by isoperimetric inequality~\eqref{isoperimetric:inequality}, \begin{equation*} P_\lambda(F) \ge n |B^\lambda|^{\frac{1}{n}} m^{\frac{n - 1}{n}}. \end{equation*}
We establish that \begin{equation*} \begin{split} P_\lambda(E^-) & \le P_\lambda(E) - P_\lambda(F) + 2  \sigma(b^-) \\ & \le P_\lambda(E) - P_\lambda(F) + \frac{1}{2}n |B^\lambda|^{\frac{1}{n}} m^{\frac{n - 1}{n}} \\ & \le P_\lambda(E) - \frac{1}{2} n |B^\lambda|^{\frac{1}{n}}m^{\frac{n - 1}{n}}. \end{split} \end{equation*} 
Since $E^- \subset E$, then $\rr(E^-) \le \rr(E)$, $\g(E^-) \le \g(E)$ and we deduce~\eqref{3.1bis:np}. \end{proof}

We now prove a uniform boundedness result.
\begin{lemma}\label{lemma3.3:np} Let $g$ be $\rr$-admissible and $q$-growing and let $G$ be $\g$-admissible. Let $\epsilon_1,\epsilon_2 >0$. 
For every $m \in (0, \infty)$ there exist $R > 0$ and $\bar h \in \N$, depending on $n$, $\lambda$, $m$, $\epsilon_1$, $\epsilon_2$, $g$, $G$ and $q$, such that \begin{equation*} \inf\left\{\ff_\epsilon^\lambda(\Omega) \st \Omega \subset \R^n \setminus H,\, |\Omega| = m\right\} \ge \inf\left\{\tilde\ff_{\epsilon, \bar h}^{\lambda, R}(\Omega) \st \Omega \subset \R^n\setminus H,\, |\Omega| = m\right\}, \end{equation*} where \begin{equation*} \tilde\ff_{\epsilon, \bar h}^{\lambda, R}(\Omega) := \inf \left\{\sum_{i = 1}^{\bar h} \ff_\epsilon^\lambda(\Omega^i) \st \Omega = \cup_{i = 1}^{\bar h}\Omega^i, \, \Omega^i \cap \Omega^j = \emptyset, \, {\rm diam} \,\Omega^i \le R \quad \forall 1 \le i \neq j \le \bar h\right\}. \end{equation*} \end{lemma}
\begin{proof} Let $M(n, \lambda, m, \epsilon_1, \epsilon_2, g, G, q) \in \N$ be a natural number to be determined later and let us denote $\bar m = m/M$.
Let $E \subset R^n \setminus H$ be a bounded set with $|E| = m$ and \begin{equation}\label{3.3:np} \ff_\epsilon^\lambda(E) \le \inf\left\{\ff_\epsilon^\lambda(\Omega) \st \Omega \in \mathcal{A}\right\} + \frac{n |B^\lambda|^{\frac{1}{n}}}{3}\left(\frac{m}{M^2}\right)^{\frac{n - 1}{n}}. \end{equation} 
This is possible since the infimum is reached by a sequence of bounded sets.
Let $t_0 < t_1 < \dots < t_{M - 1} < t_M$ be real numbers such that \begin{equation*} \left|E \cap \left((t_i, t_{i + 1}) \times \R^{n - 2} \times (0, \infty)\right)\right| = \bar m, \end{equation*} for every $0 \le i \le M - 1$ and let $L(n, \lambda, \bar m)$ be given by \cref{lemma3.2:np}.
For every $0 \le i \le M - 1$ let us define the interval $I_i$ in the following way.
If $t_{i + 1} - t_i \le 2L$ we set $I_i = \emptyset$, otherwise we apply \cref{lemma3.2:np} with $a = t_i$ and $b = t_{i + 1}$ and we set $I_i = [a^+, b^-]$.
If $m_i = |E \cap (I_i \times \R^{n - 2} \times (0, \infty))|$, then \begin{equation}\label{3.4:np} m_i \le \frac{m}{M^2}. \end{equation}
Indeed, if $I_i = \emptyset$, then~\eqref{3.4:np} is clearly true.
If $I_i \neq \emptyset$, we set \begin{equation*} E' = \alpha \left(E \setminus \left(I_i \times \R^{n - 2} \times (0, \infty)\right)\right), \end{equation*} with $\alpha = \left(\frac{m}{m - m_i}\right)^{\frac{1}{n}}$.
Note that $\frac{m_i}{m} \le \frac{1}{M}$ by construction.
By \cref{lemma3.1:npbis} and~\eqref{3.1:np} we have \begin{equation*} \begin{split} \ff_\epsilon^\lambda(E') & \le \left(\frac{m}{m - m_i}\right)^{2 + \frac{q}{n}}\ff_\epsilon^\lambda\left(E \setminus (I_i \times \R^{n - 2} \times (0, \infty))\right) \\ & \le \left(\frac{1}{1 - \frac{m_i}{m}}\right)^{2 + \frac{q}{n}} \left(\ff_\epsilon^\lambda(E) - \frac{1}{2}n |B^\lambda|^{\frac{1}{n}}m_i^{\frac{n - 1}{n}}\right). \end{split} \end{equation*}
Moreover, if $M$ is large enough,
\begin{equation}\label{3.4bis:np} \begin{split} \ff_\epsilon^\lambda(E') & \le \left(1 + \left(3 + \frac{q}{n}\right)\frac{m_i}{m}\right) \left(\ff_\epsilon^\lambda(E) - \frac{1}{2}n |B^\lambda|^{\frac{1}{n}}m_i^{\frac{n - 1}{n}}\right) \\ & \le\ff_\epsilon^\lambda(E) - \frac{1}{3} n |B^\lambda|^{\frac{1}{n}}m_i^{\frac{n - 1}{n}}. \end{split} \end{equation}
Estimates~\eqref{3.3:np} and~\eqref{3.4bis:np} imply that \begin{equation*} \ff_\epsilon^\lambda(E') \le \inf\left\{\ff_\epsilon^\lambda(\Omega) \st \Omega \in \mathcal{A}\right\} + \frac{n |B^\lambda|^{\frac{1}{n}}}{3}\left(\frac{m}{M^2}\right)^{\frac{n - 1}{n}} - \frac{n |B^\lambda|^{\frac{1}{n}}}{3}m_i^{\frac{n - 1}{n}}, \end{equation*} and, since $|E'| = m$,~\eqref{3.4:np} holds.
\par Let \begin{equation*} \tilde E = E \setminus \left(\bigcup_{i = 0}^{M - 1} I_i \times \R^{n - 2} \times (0, \infty)\right) \end{equation*} and $\mu = \sum_{i = 0}^{M - 1}m_i$, so that $\left|\tilde E\right| = m - \mu$.
By \cref{lemma3.2:np} and the subadditivity of power function with exponent less than $1$ we get \begin{equation*} \begin{split} \ff_\epsilon^\lambda(\tilde E) & \le \ff_\epsilon^\lambda(E) - \frac{1}{2}n |B^\lambda|^{\frac{1}{n}}\sum_{i = 0}^{M - 1} m_i^{\frac{n - 1}{n}} \le \ff_\epsilon^\lambda(E) - \frac{1}{2} n |B^\lambda|^{\frac{1}{n}} \mu^{\frac{n - 1}{n}}. \end{split} \end{equation*}
The set $F := \left(\frac{m}{m - \mu}\right)^{\frac{1}{n}}\tilde E$ has volume $m$.
We can use \cref{lemma3.1:npbis} to obtain \begin{equation*} \begin{split} \ff_\epsilon^\lambda(F) & \le \left(\frac{m}{m - \mu}\right)^{2 + \frac{q}{n}} \ff_\epsilon(\tilde E) \le \left(\frac{m}{m - \mu}\right)^{2 + \frac{q}{n}} \left(\ff_\epsilon^\lambda(E) - \frac{1}{2}n |B^\lambda|^{\frac{1}{n}}\mu^{\frac{n - 1}{n}}\right). \end{split} \end{equation*}
If $\mu$ is small enough, which happens as soon as $M$ is large enough thanks to \eqref{3.4:np}, we deduce \begin{equation*} \begin{split} \ff^\lambda_\epsilon(F) & \le \left(\frac{m}{m - \mu}\right)^{2 + \frac{q}{n}} \left(\ff_\epsilon^\lambda(E) - \frac{1}{2}n |B^\lambda|^{\frac{1}{n}}\mu^{\frac{n - 1}{n}}\right) 
\\ & \le 
\ff^\lambda_\epsilon(E) + C \mu \left(\inf\left\{\ff_\epsilon^\lambda(\Omega) \st \Omega \in \mathcal{A}\right\} + c\left(\frac{m}{M^2}\right)^{\frac{n-1}{n}}\right) + (1+C\mu) \left( - \frac{1}{2}n |B^\lambda|^{\frac{1}{n}}\mu^{\frac{n - 1}{n}}\right) \\ & \le \ff^\lambda_\epsilon(E). \end{split} \end{equation*}
Note that $\tilde E$ is the union of at most $M + 1$ sets, each contained in a slab having width at most equal to $2L$ by \cref{lemma3.2:np}.
In particular, $F$ is the union of at most $M + 1$ parts and each of them has horizontal width at most equal $3L$.
\par If we repeat the arguments in the remaining directions, with care to apply also \cref{lemma3.2bis:np} in the $n$-th direction, we get the boundedness of the pieces in all the $n$ directions.
Then there exist $R \in (0, \infty)$, $\bar h \in \N$ and $G \subset \R^n$ such that $|G| = m$, $G = \cup_{i = 1}^{\bar h}G_i$, $G_i \cap G_j = \emptyset$, diam $G_i \le R$ and $\ff_\epsilon^\lambda(G) \le \ff^\lambda_\epsilon(E)$.
Finally \begin{equation*} \ff_\epsilon^\lambda(G) \ge \sum_{i = 1}^{\bar h}\ff_\epsilon^\lambda(G_i) \ge \tilde\ff_{\epsilon, \bar h}^{\lambda, R}(G). \end{equation*} \end{proof}

Now we are ready to prove \cref{generalized:existence}.

\begin{proof}[Proof of \cref{generalized:existence}] We begin by proving the existence of $h' \in \N$ and of a sequence $\{G_i\}_{i \in \N} \subset \mathcal{A}$ such that \begin{equation}\label{3.5:np} \inf\left\{\tilde\ff_\epsilon^\lambda(\Omega) \st \Omega \in \mathcal{A}\right\} = \lim_{i \to \infty}\tilde\ff^\lambda_{\epsilon, h'}(G_i). \end{equation} 
Let $h'(n, \lambda, \epsilon_1, \epsilon_2, g, G, q)$ be an integer to be determined later and consider a sequence $\{E_i\}_{i \in \N} \subset \mathcal{A}$ such that \begin{equation}\label{3.6:np} K := \inf\left\{\tilde\ff_\epsilon^\lambda(\Omega) \st \Omega \in \mathcal{A}\right\} = \lim_{i \to \infty}\tilde\ff_\epsilon^\lambda(E_i). \end{equation}
For every $i \in \N$ let $h(i) \in \N$ such that there exists a subdivision $E_i = E_i^1 \cup E_i^2 \cup \dots \cup E_i^{h(i)}$ with \begin{equation}\label{3.7:np} \tilde\ff_\epsilon^\lambda(E_i) > \left(1 - \frac{1}{i + 1}\right) \sum_{j = 1}^{h(i)}\ff_\epsilon^\lambda(E_i^j). \end{equation}
Without loss of generality, we can assume $h(i) \to \infty$, so that $h(i) > h'$ for $i$ large enough.
Let us fix a generic $i \in \N$.
For simplicity of notation, let us denote $h = h(i)$ and $m_j = |E_i^j|$ for every $1 \le j \le h$.
Let us also assume, without loss of generality, that $m_j$ is decreasing with respect to $j$.
By~\eqref{3.7:np} we get
\begin{equation*} \begin{split} \tilde\ff_\epsilon^\lambda(E_i) & \ge \frac{1}{2} \sum_{j = 1}^hP_\lambda(E_i^j) \ge \frac{1}{2}\sum_{j = 1}^h n |B^\lambda|^{\frac{1}{n}}m_j^{\frac{n - 1}{n}} \ge \frac{1}{2\sqrt[n]{m_1}} \sum_{j = 1}^h n |B^\lambda|^{\frac{1}{n}} m_j = \frac{1}{2\sqrt[n]{m_1}} n |B^\lambda|^{\frac{n + 1}{n}}. \end{split} \end{equation*}
If $i$ is large enough, by~\eqref{3.6:np} we deduce that \begin{equation*}\label{costante:crescita1} m_1 \ge \left(\frac{n |B^\lambda|^{\frac{n + 1}{n}}}{4 K}\right)^n. \end{equation*}
For every such $i$, we define \begin{equation*} G_i = \alpha \bigcup_{j = 1}^{h'}  E_i^j, \end{equation*} with \begin{equation*}\label{costante:crescita2} \alpha = \left(\frac{|B^\lambda|}{|B^\lambda| - \sum_{j > h'}m_j}\right)^{\frac{1}{n}} \le 1 + c_1\sum_{j = h' + 1}^h m_j, \end{equation*} where $c_1$ is a constant depending on $n$, $\lambda$ and $K$ (that is on $n$, $\lambda$, $\epsilon_1$, $\epsilon_2$, $g$ and $G$).
Note that also $G_i$ belongs to $\mathcal{A}$.
By \cref{lemma3.1:npbis} we deduce \begin{equation}\label{3.9:np} \begin{split} \tilde\ff_{\epsilon, h'}^\lambda(G_i) & \le \sum_{j = 1}^{h'}\ff_\epsilon^\lambda(\alpha E_i^j) \le \alpha^{2n + q} \sum_{j = 1}^{h'}\ff_\epsilon(E_i^j) \\ & \le \left(1 + c_2(n, \lambda, \epsilon_1, \epsilon_2, g, G, q)\sum_{j = h' + 1}^h m_j\right)\sum_{j = 1}^{h'}\ff_\epsilon^\lambda(E_i^j). \end{split} \end{equation}
By~\eqref{3.7:np} we get \begin{equation*} \begin{split} \tilde\ff_\epsilon^\lambda(E_i) & > \left(1 - \frac{1}{i + 1}\right) \sum_{j = 1}^h \ff_\epsilon^\lambda(E_i^j) \\ & \ge \left(1 - \frac{1}{i + 1}\right) \sum_{j = 1}^{h'}\ff_\epsilon^\lambda(E_i^j) + \frac{1}{2} \sum_{j = h' + 1}^hP_\lambda(E_i^j) \\ & \ge \left(1 - \frac{1}{i + 1}\right) \sum_{j = 1}^{h'}\ff_\epsilon^\lambda(E_i^j) + \frac{1}{2}n |B^\lambda|^{\frac{1}{n}}\sum_{j = h' + 1}^hm_j^{\frac{n - 1}{n}}. \end{split} \end{equation*}
If $i$ is large enough, by~\eqref{3.6:np} and~\eqref{3.7:np} \begin{equation*} \sum_{j = 1}^{h'}\ff_\epsilon^\lambda(E_i^j) \le 2K. \end{equation*}
By~\eqref{3.9:np} \begin{equation}\label{3.10:np} \tilde\ff_{\epsilon, h'}^\lambda(G_i) - \tilde\ff_{\epsilon, h}^\lambda(E_i) \le 2K \left(c_2(n, \lambda, \epsilon_1, \epsilon_2, g, G, q)\sum_{j = h' + 1}^h m_j + \frac{1}{i + 1}\right) - \frac{1}{2}n |B^\lambda|^{\frac{1}{n}}\sum_{j = h' + 1}^h m_j^{\frac{n - 1}{n}}. \end{equation}
Now we can define $h' \in \N$ so that \begin{equation*} h' \ge \left(\frac{4K c_2(n, \lambda, \epsilon_1, \epsilon_2, g, G, q)}{n}\right)^n. \end{equation*}
Since $m_j \le \frac{|B^\lambda|}{h'}$ for $j > h'$, if $i$ is large enough we get from~\eqref{3.10:np} \begin{equation*} \tilde\ff_{\epsilon, h'}^\lambda(G_i) \le \tilde\ff_{\epsilon, h}^\lambda(E_i) + \frac{2K}{i + 1}. \end{equation*}
By~\eqref{3.6:np} we finally deduce that $\{G_i\}_{i \in \N} \subset \mathcal{A}$ satisfies~\eqref{3.5:np}.

\medskip

Let us now show that for every $m > 0$ there exist $\bar h \in \N$, a bounded set $E$ with $|E| = m$ and a subdivision $E = \bigcup_{k = 1}^{\bar h} E^k$ such that \begin{equation} \label{3.11:np} \tilde\ff_\epsilon^\lambda(E) \le \sum_{k = 1}^{\bar h}\ff_\epsilon^\lambda(E^k) \le \inf\left\{\ff_\epsilon^\lambda(\Omega) \st \Omega \in \R^n \setminus H, |\Omega| = m\right\}. \end{equation} 
Let $R$ and $\bar h$ be as in \cref{lemma3.3:np}. By \cref{lemma3.3:np} there is a sequence of sets $\{\Omega_i\}_{i \in \N}$ of volume $m$ such that \begin{equation}\label{3.12:np} \inf\left\{\ff_\epsilon^\lambda(\Omega) \st \Omega \subset \R^n \setminus H, |\Omega| = m\right\} \ge \lim_{i \to + \infty} \tilde\ff_{\epsilon, \bar h}^{\lambda, R}(\Omega_i), \end{equation} where $\tilde\ff_{\epsilon, \bar h}^{\lambda, R}$ is defined in \cref{lemma3.3:np}.
For every $i \in \N$ there exists a partition $\Omega_i = \Omega_i^1 \cup \Omega_i^2 \cup \dots \cup \Omega_i^{\bar h}$ with diam$(\Omega_i^j) \le R$ and \begin{equation}\label{3.13:np} \sum_{j = 1}^{\bar h}\ff_\epsilon^\lambda(\Omega_i^j) \le \tilde\ff_{\epsilon, \bar h}^{\lambda, R}(\Omega_i) + \frac{1}{i}. \end{equation}
Up to a subsequence there exist $m_j \in (0, \infty)$, with $1 \le j \le \bar h$, such that \begin{equation*} m_j = \lim_{i \to \infty} |\Omega_i^j| \qquad \forall 1 \le j \le \bar h, \qquad
\qquad
m = \sum_{j = 1}^{\bar h} m_j. \end{equation*}
Let us fix $1 \le k \le \bar h$ and consider the sets $\{\Omega_i^k\}_{i \in \N}$.
Since their diameters are uniformly bounded by $R$, up to translations we can assume that all the $\Omega_i^k$ are pairwise disjoint and contained in a fixed ball with radius $R$.
Therefore the characteristic functions $f_i = \chi_{\Omega_i^k}$ have uniformly bounded supports and are bounded in $BV$.
Up to a subsequence, we can assume that $f_i$ is weakly$^*$ convergent in $BV$, and in particular strongly convergent in $L^1$, to a certain function $f$.
Then $f$ is the characteristic function of a bounded set $E^k$ with volume $m_k$.
By the lower-semicontinuity of the perimeter under weak$^*$ $BV$-convergence and the continuity of $\rr$ and $\g$ under strong $L^1$ convergence, we obtain that \begin{equation}\label{3.14:np} \ff_\epsilon^\lambda(E^k) \le \liminf_{i \to \infty}\ff_\epsilon^\lambda(\Omega_i^k). \end{equation}
Up to a translation we can assume that the sets $E^k$ are pairwise disjoint.
In particular the set $E = \cup_{k = 1}^{\bar h} E^k$ is bounded with $|E| = m$.
By~\eqref{3.12:np},~\eqref{3.13:np} and~\eqref{3.14:np} we get \begin{equation*} \begin{split} \tilde\ff_\epsilon^\lambda(E) & \le \sum_{k = 1}^{\bar h}\ff_\epsilon^\lambda(E^k) \le \sum_{k = 1}^{\bar h} \liminf_{i \to \infty}\ff_\epsilon^\lambda(\Omega_i^k) \le \liminf_{i \to \infty}\sum_{k = 1}^{\bar h} \ff_\epsilon^\lambda(\Omega_i^k) \\ & \le \liminf_{i \to \infty}\tilde\ff_{\epsilon, \bar h}^{\lambda, R}(\Omega_i) \le \inf\left\{\ff_\epsilon^\lambda(\Omega) \st \Omega \subset \R^n \setminus H, \; |\Omega| = m\right\}, \end{split} \end{equation*} so~\eqref{3.11:np} is proved.

\medskip

We can now conclude the proof of the theorem. Let $\{G_i\}_{i \in \N}$ as in~\eqref{3.5:np} and let us consider a subdivision $G_i = G_i^1 \cup G_i^2 \cup \dots \cup G_i^{h'}$ such that \begin{equation}\label{3.15:np} \inf\left\{\tilde\ff_\epsilon^\lambda(\Omega) \st \Omega \in \mathcal{A}\right\} = \lim_{i \to \infty} \sum_{j = 1}^{h'}\ff_\epsilon^\lambda(G_i^j). \end{equation}
Up to a subsequence there exist $\mu_j > 0$, $1 \le j \le h'$, such that \begin{equation*} \mu_i = \lim_{i \to \infty} |G_i^j| \qquad \forall 1 \le j \le h', \; |B^\lambda| = \sum_{j = 1}^{h'}\mu_j. \end{equation*}
Let \begin{equation*} K_j := \inf\left\{\ff_\epsilon^\lambda(\Omega) \st \Omega \subset \R^n \setminus H, |\Omega| = \mu_j\right\}. \end{equation*}
By~\eqref{3.15:np} \begin{equation}\label{3.16:np} \inf\left\{\tilde\ff_\epsilon^\lambda(\Omega) \st \Omega \in \mathcal{A}\right\} = \sum_{j = 1}^{h'}K_j. \end{equation}
By~\eqref{3.11:np} for every $1 \le j \le h'$ there exist $\bar h(j) \in \N$, a bounded set $E_j \subset \R^n \setminus H$ with $|E_j| = \mu_j$ and a subdivision in pairwise disjoint sets $E_j = \bigcup_{k = 1}^{\bar h(j)} E_{j, k}$ such that \begin{equation}\label{3.17:np} \tilde\ff_\epsilon^\lambda(E_j) \le \sum_{k = 1}^{\bar h(j)} \ff_\epsilon^\lambda(E_{j, k}) \le K_j. \end{equation}
Since the sets $E_j$ are bounded, up to translations we can assume that the set $E = \cup_{j = 1}^{h'} E_j$ has volume $|B^\lambda|$.
Therefore $E$ is the disjoint union of all the sets $E_{j, k}$ with $1 \le j \le h'$ and $1 \le k \le \bar h(j)$.
Let us denote these sets as $E^l$ with $1 \le l \le h$ and $h = \sum_{j = 1}^{h'}\bar h(j)$.
By~\eqref{3.16:np} and~\eqref{3.17:np} we deduce that \begin{equation*} \tilde\ff_\epsilon^\lambda(E) \le \sum_{l= 1}^h \ff_\epsilon^\lambda(E^l) \le \sum_{j = 1}^{h'} K_j = \inf \left\{\tilde\ff_\epsilon^\lambda(\Omega) \st \Omega \in \mathcal{A}\right\}, \end{equation*} that is $E$ is a minimizer of $\tilde\ff_\epsilon^\lambda$ and the subdivision $E = \bigcup_{l=1}^h E^l$ is optimal.
\par The proof of~\eqref{1.3:np} for a given $1 \le \bar j \le h$ easily follows as in the proof of \cite[Proposition 1.2]{NovagaPratelli}. \end{proof}

Now we are ready to prove \cref{theorem:3}.
\begin{proof}[Proof of \cref{theorem:3}] With the notation of \cref{theorem:3}, if $\Omega \subset \R^n \setminus H$ is a measurable set with $|\Omega| = m$, let \begin{equation*} \ff^\lambda(\Omega) = P_\lambda(\Omega) + \int_\Omega \int_\Omega g(y - x) \de y \de x + \int_\Omega G(x_n) \de x \end{equation*} and  \begin{equation*}\tilde\ff^\lambda(\Omega) := \inf_{h \in \N}\tilde\ff^\lambda_{h}(\Omega), \end{equation*} where \begin{equation*} \tilde\ff^\lambda_{h}(\Omega) := \inf \left\{\sum_{i = 1}^h \ff^\lambda(\Omega^i) \st \Omega = \bigcup_{i = 1}^h \Omega^i, \Omega^i \cap \Omega^j = \emptyset \quad {\rm for} \; 1 \le i \neq j \le h\right\}. \end{equation*} 
If $F \subset \R^n \setminus H$ has measure $|B^\lambda|$ and $\bar \epsilon = \frac{m^{\frac{1}{n}}}{|B^\lambda|^{\frac{1}{n}}}$, the set $\tilde F := \bar \epsilon F$ has volume $m$ and by \cref{remark:scaling} there exist $\tilde g$ $\rr$-admissible and $\tilde G$ $\g$-admissible such that \begin{equation*} \ff^\lambda (\tilde F) = \bar \epsilon^{n - 1} \left(P_\lambda(F) + \bar \epsilon^{n + 1}\int_F\int_F \tilde g(y - x) \de y \de x + \bar\epsilon\int_F \tilde G(x_n) \de x\right) =: \bar \epsilon^{n - 1} \ff^\lambda_{\bar\epsilon}(F). \end{equation*}
Note that $\tilde\ff^\lambda(\tilde F) = \bar \epsilon^{n - 1} \tilde\ff_{\bar\epsilon}^\lambda(F)$ and that, if $g$ is $q$-growing, also $\tilde g$ is $q$-growing (see \cref{remark:scaling}).
By \cref{generalized:existence} there exists $E \subset \R^n \setminus H$ with $|E| = |B^\lambda|$ which minimizes $\tilde \ff^\lambda_{\bar\epsilon}$.
Then the set $\tilde E := \bar\epsilon E$ minimizes $\tilde\ff^\lambda$ among sets with volume $m$ and \cref{theorem:3} easily follows. \end{proof}

\section*{Acknowledgements} 
The author is member of INdAM - GNAMPA. 
The author is grateful to Marco Pozzetta for many suggestions and for stimulating discussions. 
The author also thanks the referee for all valuable comments helping to concretely improve exposition of the results.

\printbibliography[title={References}]
\end{document}